\pdfoutput=1

\documentclass{amsproc}

\usepackage{verbatim}
\usepackage{xcolor}
\usepackage{tikz}
\usepackage{tikz-cd}
\usepackage{arydshln}
\usepackage{caption}
\usepackage{subcaption}
\usepackage{graphicx}
\usepackage[hyphens]{url}
\usepackage{amsmath}
\usepackage{bbm}
\usepackage{hyperref}

\theoremstyle{theorem}
\newtheorem{theorem}{Theorem}[section]
\newtheorem{proposition}[theorem]{Proposition}
\newtheorem{lemma}[theorem]{Lemma}
\newtheorem{corollary}[theorem]{Corollary}

\theoremstyle{definition}
\newtheorem{definition}[theorem]{Definition}

\theoremstyle{remark}
\newtheorem{remark}[theorem]{Remark}

\numberwithin{equation}{section}

\begin{document}

\title[Counting square-tiled surfaces with prescribed foliations]{Counting square-tiled surfaces with prescribed real and imaginary foliations and connections to Mirzakhani's asymptotics for simple closed hyperbolic geodesics}

\author{Francisco Arana--Herrera}
\address{Department of Mathematics, Stanford University, 450 Serra Mall
	Building 380, Stanford, CA 94305-2125, USA}
\email{farana@stanford.edu}

\begin{abstract}
We show that the number of square-tiled surfaces of genus $g$, with $n$ marked points, with one or both of its horizontal and vertical foliations belonging to fixed mapping class group orbits, and having at most $L$ squares, is asymptotic to $L^{6g-6+2n}$ times a product of constants appearing in Mirzakhani's count of simple closed hyperbolic geodesics. Many of the results in this paper reflect recent discoveries of Delecroix, Goujard, Zograf, and Zorich, but the approach considered here is very different from theirs. We follow conceptual and geometric methods inspired by Mirzakhani's work. 
\end{abstract}

\maketitle

\thispagestyle{empty}

\tableofcontents

\section{Introduction}

On any closed hyperbolic surface of genus two, the ratio of the number of separating versus non-separating simple closed geodesics of length $\leq L$ converges to $1/48$ as $L \to \infty$. This surprisingly precise result, together with generalizations to arbitrary complete, finite volume hyperbolic surfaces, was proved by Mirzakhani in \cite{Mir08b}. Recently, in \cite{DGZZ19}, Delecroix, Goujard, Zograf, and Zorich proved analogous counting results for square-tiled surfaces of finite type and a posteriori related these results to Mirzakhani's asymptotics by comparing explicit formulas for the asymptotics of each counting problem. In particular, their work shows that the ratio of the number of square-tiled surfaces of genus two having one horizontal cylinder with separating versus non-separating core curve and at most $L$ squares converges to $1/48$ as $L \to \infty$.\\

The main goal of this paper is to establish a direct connection between these counting results. The novelty of our approach is to use the parametrization of non-zero, integrable, holomorphic quadratic differentials in terms of filling pairs of measured geodesic laminations to establishing a direct connection between these counting problems and to use a famous result of Hubbard, Masur, and Gardiner, which we recall as Theorem \ref{hubbard_masur_slice} below. This approach requires understanding the stabilizers of the corresponding mapping class group actions and certain no escape of mass properties of related counting measures.\\

\textit{Main results.} Let $S_{g,n}$ be a connected, oriented, smooth surface of genus $g$ with $n$ punctures and negative Euler characteristic. Let $\text{Mod}_{g,n}$ be the mapping class group of $S_{g,n}$. Let $Q\mathcal{M}_{g,n}$ be the moduli space of non-zero, integrable, holomorphic quadratic differentials on $S_{g,n}$. Let $Q\mathcal{M}_{g,n}(\mathbf{Z}) \subseteq Q\mathcal{M}_{g,n}$ be the subset of all square-tiled surfaces in $Q\mathcal{M}_{g,n}$. Every square-tiled surface $[(X,q)] \in Q\mathcal{M}_{g,n}(\mathbf{Z})$ is horizontally and vertically periodic. The core curves of the horizontal and vertical cylinders of $[(X,q)]$ define integral muti-curves $\gamma_1 := \textbf{Re}([(X,q)])$ and $\gamma_2 := \textbf{Im}([(X,q)])$ on $X$, respectively. Two multi-curves on diffeomorphic surfaces are of the same topological type if there is a diffeomorphism between the surfaces carrying one multi-curve to the other. The topological type of a multi-curve $\gamma$ on a surface diffeomorphic to $S_{g,n}$ is its equivalence class $[\gamma]$ with respect to this equivalence relation. \\

Fix two integral multi-curves $\gamma_1$ and $\gamma_2$ on $S_{g,n}$. We are interested in the growth as $L \to \infty$ of the following quantity:
\begin{equation}
\label{counting_function_square_tiled_surfaces}
s(\gamma_1,\gamma_2,L) = \sum_{\substack{[(X,q)] \in Q\mathcal{M}_{g,n}(\mathbf{Z}), \\ \textbf{Re}([(X,q)]) \in [\gamma_1], \ \textbf{Im}([(X,q)]) \in [\gamma_2], \\ \text{Area}([(X,q)]) \leq L }} \frac{1}{\# \text{Aut}([(X,q)])}.
\end{equation}
More concretely, $s(\gamma_1,\gamma_2,L)$ is the automorphism weighted count of the number of square-tiled surfaces of area $\leq L$ (or equivalently, having at most $L$ squares), whose horizontal cylinders have core multi-curve of topological type $[\gamma_1]$, and whose vertical cylinders have core multi-curve of topological type $[\gamma_2]$. Notice that $s(\gamma_1,\gamma_2,L)$ is finite because there are only finitely many square-tiled surfaces with area $\leq L$.\\

One of the main results of this paper is the following theorem, which gives a precise description of the growth of $s(\gamma_1,\gamma_2,L)$ as $L \to \infty$.\\

\begin{theorem}
	\label{counting_quadratic_differentials_intro}
	For any pair of integral multi-curves $\gamma_1$ and $\gamma_2$ on $S_{g,n}$,
	\[
	\lim_{L \to \infty} \frac{s(\gamma_1,\gamma_2,L)}{L^{6g-6+2n}} = \frac{c(\gamma_1) \cdot c(\gamma_2)}{2^{2g-3+n} \cdot b_{g,n}},
	\]
	where $c(\gamma_1), c(\gamma_2) \in \mathbf{Q}_{>0}$ are the frequencies of integral multi-curves on $S_{g,n}$ of topological type $[\gamma_1], [\gamma_2]$ defined in (5.2) in \cite{Mir08b} and $b_{g,n} \in \mathbf{Q}_{>0} \cdot \pi^{6g-6+2n}$ is the constant depending only on $g$ and $n$ defined in (3.4) in \cite{Mir08b}.\\
\end{theorem}

Given an integral multi-curve $\gamma$ on $S_{g,n}$, we are also interested in the growth as $L \to \infty$ of the following quantity:
\begin{equation}
\label{counting_square_tiled_surfaces_2}
s(\gamma,*,L) = \sum_{\substack{[(X,q)] \in Q\mathcal{M}_{g,n}(\mathbf{Z}), \\ \textbf{Re}([(X,q)]) \in [\gamma], \\ \text{Area}([(X,q)]) \leq L }} \frac{1}{\# \text{Aut}([(X,q)])}.
\end{equation}
$ $ \vspace{+0.2cm}

The same arguments used in the proof of Theorem \ref{counting_quadratic_differentials_intro} give the following result.\\

\begin{theorem}
	\label{counting_quadratic_differentials_2_intro}
	For any integral multi-curve $\gamma$ on $S_{g,n}$,
	\[
	\lim_{L \to \infty} \frac{s(\gamma,*,L)}{L^{6g-6+2n}} = \frac{c(\gamma)}{2^{2g-3+n}},
	\]
	where $c(\gamma) \in \mathbf{Q}_{>0}$ is the frequency of integral multi-curves on $S_{g,n}$ of topological type $[\gamma]$ defined in (5.2) in \cite{Mir08b}.\\
\end{theorem}

Consider the following normalization of the Masur-Veech volume of $Q\mathcal{M}_{g,n}$,
\[
m_{g,n} := \lim_{L \to \infty} \frac{1}{L^{12g-12+4n}} \sum_{\substack{[(X,q)] \in Q\mathcal{M}_{g,n}(\mathbf{Z}), \\ \text{Area}([(X,q)]) \leq L}} \frac{1}{\# \text{Aut}([(X,q)])}.
\]
Theorem 1.19 in \cite{DGZZ16} shows that the horizontal and vertical cylinder decompositions of square-tiled surfaces are uncorrelated. This result together with Theorem \ref{counting_quadratic_differentials_2_intro} give the following alternative version of Theorem \ref{counting_quadratic_differentials_intro}.\\

\begin{theorem}
	\label{counting_quadratic_differentials_v2_intro}
	For any pair of integral multi-curves $\gamma_1$ and $\gamma_2$ on $S_{g,n}$,
	\[
	\lim_{L \to \infty} \frac{s(\gamma_1,\gamma_2,L)}{L^{6g-6+2n}} = \frac{c(\gamma_1) \cdot c(\gamma_2)}{2^{4g-6+2n} \cdot m_{g,n}},
	\]
	where $c(\gamma_1), c(\gamma_2) \in \mathbf{Q}_{>0}$ are the frequencies of integral multi-curves on $S_{g,n}$ of topological type $[\gamma_1], [\gamma_2]$ defined in (5.2) in \cite{Mir08b}.\\
\end{theorem}

Combining Theorems \ref{counting_quadratic_differentials_intro} and \ref{counting_quadratic_differentials_v2_intro} gives the following formula for the Masur-Veech volumes $m_{g,n}$.\\

\begin{corollary}
	\label{masur_veech_volume}
	For every $g,n \in \mathbf{Z}_{\geq 0}$ with $2 -2g- n < 0$,
	\[
	m_{g,n} = \frac{b_{g,n}}{2^{2g-3+n}}.
	\]
\end{corollary}
$ $\vspace{+0.2cm}

\begin{remark}
	\label{zorich}
	Theorems \ref{counting_quadratic_differentials_intro}, \ref{counting_square_tiled_surfaces_2}, and \ref{counting_quadratic_differentials_v2_intro}, and Corollary \ref{masur_veech_volume} reflect recent discoveries of Delecroix, Goujard, Zograf, and Zorich; see \cite{DGZZ19}. Their work uses a completely different approach based on combinatorial methods which involve counting square-tiled surfaces by using Kontsevich's combinatorial formula for ribbon graphs, see \cite{Kon92}, and a posteriori relating the asymptotics of such count to Mirzakhani's asymptotics for simple closed hyperbolic geodesics by considering Mirzakhani's description of such asymptotics in terms of intersection numbers of tautological line bundles over moduli spaces of Riemann surfaces. An alternative version of Corollary \ref{masur_veech_volume} for the case $n = 0$ was previously proved by Mirzakhani in \cite{Mir08a}.\\
\end{remark}

\textit{Main tools.} Arguably the most important tool used in this paper is the parametrization of non-zero, integrable, holomorphic quadratic differentials on $S_{g,n}$ by their real and imaginary foliations. More precisely, let $Q\mathcal{T}_{g,n}$ be the Teichm\"uller space of marked non-zero, integrable, holomorphic quadratic differentials on $S_{g,n}$ and let $\mathcal{ML}_{g,n}$ be the space of measured geodesic lamination on $S_{g,n}$. Consider the subset $\Delta \subseteq \mathcal{ML}_{g,n} \times \mathcal{ML}_{g,n}$ of pairs of measured geodesic laminations that do not fill $S_{g,n}$. The desired parametrization is given by the map
\begin{equation}
\label{hubbard_masur_definition}
\begin{array}{c c c c}
h: & Q\mathcal{T}_{g,n} &\to & \mathcal{ML}_{g,n} \times \mathcal{ML}_{g,n} - \Delta \\
\end{array}
\end{equation}
which assigns to every marked non-zero, integrable, holomorphic quadratic differential on $S_{g,n}$ its real and imaginary foliations interpreted as measured geodesic laminations on $S_{g,n}$. Following the conventions of Lindenstrauss and Mirzakhani in \cite{LM08}, we refer to this map as the \textit{Hubbard-Masur map}.\\

\begin{theorem}
	\label{hubbard_masur_map}
	The Hubbard-Masur map $h \colon Q\mathcal{T}_{g,n} \to \mathcal{ML}_{g,n} \times \mathcal{ML}_{g,n} - \Delta$ is a mapping class group equivariant homeomorphism sending marked square-tiled surfaces to pairs of filling integral multi-curves and sending area of quadratic differentials to geometric intersection number of measured geodesic laminations.\\
\end{theorem}

Fix a measured geodesic lamination $\lambda \in \mathcal{ML}_{g,n}$. Let $\mathcal{ML}_{g,n}(\lambda) \subseteq \mathcal{ML}_{g,n}$ be the open subset of all measured geodesic laminations that together with $\lambda$ fill $S_{g,n}$, that is 
\[
\mathcal{ML}_{g,n}(\lambda) := \{ \mu \in \mathcal{ML}_{g,n} \ | \ (\lambda, \mu) \in \mathcal{ML}_{g,n} \times \mathcal{ML}_{g,n} - \Delta\}.
\]
Let $\mathcal{T}_{g,n}$ be the Teichm\"uller space of marked punctured complex structures on $S_{g,n}$ and $p \colon Q\mathcal{T}_{g,n} \to \mathcal{T}_{g,n}$ be the natural projection of the bundle of non-zero, integrable, holomorphic quadratic differentials on $\mathcal{T}_{g,n}$. The inverse of the Hubbard-Masur map induces a \textit{Hubbard-Masur slice} $h_\lambda \colon \mathcal{ML}_{g,n}(\lambda)  \to \mathcal{T}_{g,n} $ given by the following composition:
\[
\begin{array}{ c c c c c c c c}
h_\lambda \colon & \mathcal{ML}_{g,n}(\lambda) & \to &\mathcal{ML}_{g,n} \times \mathcal{ML}_{g,n} - \Delta  & \to & Q\mathcal{T}_{g,n} & \to & \mathcal{T}_{g,n} \\
&\mu & \mapsto & (\lambda,\mu) & \mapsto & h^{-1}(\lambda,\mu) & \mapsto  & p(h^{-1}(\lambda,\mu)).
\end{array}
\]
$ $ \vspace{+0.2cm}

The following deep theorem, originally due to Hubbard and Masur in the case $n = 0$, see \cite{HM79}, and extended by Gardiner and Masur to the case $n > 0$, see \cite{Gar87} and \cite{GM91}, is absolutely crucial to develop the material presented in this paper; see \cite{Ker80} for an alternative proof by Kerckhoff using Jenkins-Strebel differentials and \cite{Wol96} for an elementary proof by Wolf using harmonic maps.\\

\begin{theorem}
	\label{hubbard_masur_slice}
	For every $\lambda \in \mathcal{ML}_{g,n}$, the Hubbard-Masur slice $h_\lambda \colon \mathcal{ML}_{g,n}(\lambda) \to \mathcal{T}_{g,n}$ is a $\text{Stab}(\lambda)$-equivariant homeomorphism.\\
\end{theorem}

Mirzakhani's curve counting results in \cite{Mir08b} are also an important tool in this paper as they provide suitable vocabulary and play an important role in the proof of Theorem \ref{counting_quadratic_differentials_intro}. Let $\mu_{\text{Thu}}$ be the Thurston measure on $\mathcal{ML}_{g,n}$. Given a rational multi-curve $\gamma$ on $S_{g,n}$ and $L > 0$, consider on $\mathcal{ML}_{g,n}$ the counting measure
\begin{equation}
\label{curve_counting_measures}
\mu_{\gamma}^L := \frac{1}{L^{6g-6+2n}} \sum_{\alpha \in \text{Mod}_{g,n} \cdot \gamma} \delta_{\frac{1}{L} \cdot \alpha}.
\end{equation}
The proof of Theorem \ref{counting_quadratic_differentials_intro} relies on the following measure convergence result, which is Theorem 1.3 in \cite{Mir08b}.\\

\begin{theorem}
	\label{mirzakhani_measure_convergence_intro}
	For any rational multi-curve $\gamma$ on $S_{g,n}$,
	\[
	\lim_{L \to \infty} \mu_{\gamma}^L = \frac{c(\gamma)}{b_{g,n}} \cdot \mu_{\text{Thu}},
	\]
	where $c(\gamma) \in \mathbf{Q}_{>0}$ is the frequency of integral multi-curves on $S_{g,n}$ of topological type $[\gamma]$ defined in (5.2) in \cite{Mir08b} and $b_{g,n} \in \mathbf{Q}_{>0} \cdot \pi^{6g-6+2n}$ is the constant depending only on $g$ and $n$ defined in (3.4) in \cite{Mir08b}.\\
\end{theorem}

\textit{Sketch of proof of Theorem \ref{counting_quadratic_differentials_intro}}. Let $\gamma_1$ and $\gamma_2$ be integral multi-curves on $S_{g,n}$. 
Using Theorem \ref{hubbard_masur_map} we show that
\[
s(\gamma_1,\gamma_2,L) = \sum_{\substack{[\beta]\in \mathcal{ML}_{g,n}(\gamma_1)/\text{Stab}(\gamma_1), \\ \beta \in \text{Mod}_{g,n} \cdot \gamma_2, \\ i(\gamma_1,\beta) \leq L }} \frac{1}{\# (\text{Stab}(\gamma_1) \cap \text{Stab}(\beta))},
\]
where $i(\cdot,\cdot)$ denotes the geometric intersection number of integral multi-curves and where the stabilizers are taken with respect to the action of the mapping class group on the set of all integral multi-curves on $S_{g,n}$. Consider the counting measures $\mu_{\gamma_2}^L$ on $\mathcal{ML}_{g,n}$ defined in (\ref{curve_counting_measures}).  Let $\widehat{\mu}_{\gamma_2}^L$ be the local pushforward of the measure $\mu_{\gamma_2}^L|_{\mathcal{ML}_{g,n}(\gamma_1)}$ under the quotient map $\mathcal{ML}_{g,n}(\gamma_1) \to \mathcal{ML}_{g,n}(\gamma_1)/\text{Stab}(\gamma_1)$. These local pushforwards exist because the action of $\text{Stab}(\gamma_1)$ on $\mathcal{ML}_{g,n}(\gamma_1)$ is properly discontinuous. This follows from Theorem \ref{hubbard_masur_slice} and the proper discontinuity of the mapping class group action on $\mathcal{T}_{g,n}$. Consider the subsets 
\begin{align*}
B(\gamma_1) &:= \{\lambda \in \mathcal{ML}_{g,n}(\gamma_1) \ | \ i(\gamma_1,\lambda) \leq 1\} \subseteq \mathcal{ML}_{g,n}, \\
\widehat{B}(\gamma_1) &:= B(\gamma_1)/\text{Stab}(\gamma_1) \subseteq \mathcal{ML}_{g,n}(\gamma_1)/\text{Stab}(\gamma_1).
\end{align*}
Unraveling definitions we show that
\[
\frac{s(\gamma_1,\gamma_2,L)}{L^{6g-6+2n}} = \widehat {\mu}_{\gamma_2}^L(\widehat{B}(\gamma_1)).
\]
Theorem \ref{mirzakhani_measure_convergence_intro} implies $\widehat{\mu}_{\gamma_2}^L \to \frac{c(\gamma_2)}{b_{g,n}} \cdot \widehat{\mu}_{\text{Thu}}$ as $L \to \infty$, where $ \widehat{\mu}_{\text{Thu}}$ is the local pushforward of the measure $\mu_{\text{Thu}}|_{\mathcal{ML}_{g,n}(\gamma_1)}$ under the quotient map $\mathcal{ML}_{g,n}(\gamma_1) \to \mathcal{ML}_{g,n}(\gamma_1)/\text{Stab}(\gamma_1)$. A no escape of mass argument using Theorem \ref{hubbard_masur_map} and period coordinates shows that
\[
\widehat{\mu}_{\gamma_2}^L (\widehat{B}(\gamma_1))\to \frac{c(\gamma_2)}{b_{g,n}} \cdot \widehat{\mu}_{\text{Thu}}(\widehat{B}(\gamma_1))
\]
as $L \to \infty$. We deduce that
\begin{equation}
\label{almost_finished_proof_intro}
\frac{s(\gamma_1,\gamma_2,L)}{L^{6g-6+2n}} \to \frac{c(\gamma_2)}{b_{g,n}} \cdot \widehat{\mu}_{\text{Thu}}(\widehat{B}(\gamma_1))
\end{equation}
as $L \to \infty$. Clearly $s(\gamma_1,\gamma_2,L) = s(\gamma_2,\gamma_1,L)$. As a consequence,
\[
\frac{s(\gamma_1,\gamma_2,L)}{L^{6g-6+2n}} \to \frac{c(\gamma_1)}{b_{g,n}} \cdot \widehat{\mu}_{\text{Thu}}(\widehat{B}(\gamma_2))
\] 
as $L \to \infty$. We deduce 
\[
\frac{\widehat{\mu}_{\text{Thu}}(\widehat{B}(\gamma_1))}{c(\gamma_1)} = \frac{\widehat{\mu}_{\text{Thu}}(\widehat{B}(\gamma_2))}{c(\gamma_2)}.
\]
As this holds for all integral multi-curves $\gamma_1$ and $\gamma_2$ on $S_{g,n}$, it follows that $r_{g,n} := \frac{\widehat{\mu}_{\text{Thu}}(\widehat{B}(\gamma))}{c(\gamma)}$ is a constant depending only on $g$ and $n$ and not on the integral multi-curve $\gamma$. Explicit computations when $\gamma$ is a pair of pants decomposition of $S_{g,n}$ show that $r_{g,n} = \frac{1}{2^{2g-3+n}}$. Theorem \ref{counting_quadratic_differentials_intro} then follows from (\ref{almost_finished_proof_intro}).\\

\begin{remark}
	The explicit computation of  $\widehat{\mu}_{\text{Thu}}(\widehat{B}(\gamma))$ when $\gamma$ is a pair of pants decomposition of $S_{g,n}$ introduces a novel approach for computing Thurston volumes by reducing to a lattice counting problem in Euclidean space through the Dehn-Thurston parametrization of integral multi-curves on $S_{g,n}$.\\
\end{remark}

\begin{remark}
	The importance of Theorems \ref{counting_quadratic_differentials_intro}, \ref{counting_quadratic_differentials_2_intro}, and \ref{counting_quadratic_differentials_v2_intro} being stated in terms of multi-curve frequencies rather than Thurston volumes is that explicit formulas for computing such frequencies as a sum of Weil-Petersson volumes were provided by Mirzakhani in \cite{Mir08b}.\\
\end{remark}

\begin{remark}
	If we define the constants $b_{g,n}$ following Mirzakhani's conventions in \cite{Mir08a} rather than in \cite{Mir08b}, the powers of two in Theorem \ref{counting_quadratic_differentials_intro} and Corollary $\ref{masur_veech_volume}$ disappear.\\
\end{remark}

\begin{remark}
	There are several different conventions of what should be called a square-tiled surface but they are all essentially equivalent; we consider the one best adapted to our arguments.  See \cite{AEZ16} for other definitions.\\
\end{remark}

\textit{Survey of similar counting problems}. The study of similar counting problems in hyperbolic geometry can be traced back to Delsarte, Huber, and  Selberg, who showed that on any complete, finite volume, hyperbolic surface, the number of primitive closed geodesics of length $\leq L$ grows asymptotically like $e^L/L$ as $L \to \infty$; in particular, the asymptotic growth is independent of the topology of the surface. This result is commonly known as the \textit{prime geodesic theorem}. In his thesis, see \cite{Mar04} for an English translation, Margulis proved an analogous result for arbitrary compact Riemannian manifolds of strictly negative sectional curvaturature; in this case the asymptotic growth is of the form $e^{hL}/hL$, where $h$ is the volume entropy of the manifold. In \cite{Mir08b}, Mirzakhani proved that on any complete, finite volume, hyperbolic surface, the asymptotic growth of the number of simple closed geodesics of length $\leq L$ is polynomial, in constrast to exponential, of degree which depends only on the topology of the surface. Moreover, Mirzakhani studied the asymptotics of each mapping class group orbit separately, giving precise formulas in terms of Weil-Petersson volumes for the leading coefficient of the polynomial describing the asymptotic growth of the number of simple closed curves of hyperbolic length $\leq L$ in a given mapping class group orbit. In \cite{Riv12}, Rivin extended Mirzakhani's results to closed curves with one self-intersection. In \cite{Mir16}, Mirzakhani showed that the same asymptotic growth holds for the mapping class group orbit of any closed curve, without restrictions on the number of intersections. Erlandsson, Parlier, and Souto, see \cite{ES16}, \cite{Erl16}, and \cite{EPS16}, extended Mirzakhani's results to general length functions of closed curves, not only hyperbolic length, that extend to the space of geodesic currents; see \cite{EU18} for a unified discussion. More recently, Rafi and Souto, see \cite{RS17}, proved analogous counting results for mapping class group orbits of arbitrary filling geodesic currents. As part of the same work, a related lattice counting result for mapping class group orbits of points in the Teichm\"uller space of marked hyperbolic structures on a closed, connected, oriented surface is proved. The counting is considered with respect to the Thurston metric. An analogous lattice counting problem for the Teichm\"uller metric instead of the Thuston metric was previously studied by Athreya, Bufetov, Eskin, and Mirzakhani in \cite{ABEM12}.\\ 

Similar counting problems in complex analysis arised originally from computations of Masur-Veech volumes of strata of Abelian differentials. Such volumes were first introduced and proved to be finite by Masur, in \cite{Mas82}, and Veech, in \cite{Vee82}. Their interest in such quantities originated from the study of interval exchange transformations. In \cite{EO01}, Eskin and  Okounkov, following ideas of Kontsevich, Masur, and Zorich, provided formulas for computing Masur-Veech volumes of strata of Abelian differentials by considering the relation of such quantities with the asymptotic growth of the number of brached covers of a torus with fixed ramification type as the degree of the cover tends to infinity. In \cite{AEZ16}, Athreya, Eskin, and Zorich computed Masur-Veech volumes of strata of quadradratic differentials on genus zero surfaces by alternative methods, thus providing an explicit expression for the leading term of the function counting associated pillowcase covers when the degree of the cover tends to infinity. More recently, Delecroix, Goujard, Zograf, and Zorich, see \cite{DGZZ16}, computed the absolute contribution of square-tiled surfaces having a single horizontal cylinder to the Masur-Veech volume of any ambient strata of Abelian differentials. In \cite{DGZZ17}, the same authors used the results in \cite{AEZ16} and \cite{DGZZ16} to derive applications to asymptotic enumeration of meanders. As pointed out in Remark \ref{zorich}, the forthcoming work \cite{DGZZ19} of the same authors proves many of the results in this paper by different methods.\\

\textit{Organization of the paper.} In Section 2 we present the background material and notation necessary to understand the proofs of Theorems \ref{counting_quadratic_differentials_intro} and \ref{counting_quadratic_differentials_2_intro}. In Section 3 we present the proofs of Theorems \ref{counting_quadratic_differentials_intro} and \ref{counting_quadratic_differentials_2_intro} in full detail. In Section 4 we compute the value of the constants $r_{g,n}$ that appear in the proof of Theorem \ref{counting_quadratic_differentials_intro} by considering the case of a pair of pants decomposition. In Section 5 we present explicit examples. \\

\textit{Acknowledgments.} The author would like to thank Alex Wright for suggesting the problem discussed in this paper and for his constant support along the development of this project. The author would also like to thank Steven Kerckhoff for his invaluable advice, patience, and encouragement.\\

\section{Background material}

\textit{Notation.} Let $g,n \geq 0$ be integers such that $2 - 2g - n < 0$. For the rest of this paper, $S_{g,n}$ will denote a connected, oriented, smooth surface of genus $g$ with $n$ punctures (and negative Euler characteristic). For $g \geq 2$ we will also use the notation $S_{g} := S_{g,0}$.\\

\textit{Teichm\"uller and moduli spaces of Riemann surfaces.} The Teichm\"uller space of $S_{g,n}$, denoted $\mathcal{T}_{g,n}$, is the space of all marked punctured complex structures on $S_{g,n}$ up to isotopy. More precisely, $\mathcal{T}_{g,n}$ is the space of pairs $(X,\phi)$, where $X$ is a punctured Riemann surface and $\phi \colon S_{g,n} \to X$ is an orientation-preserving diffeomorphism, modulo the equivalence relation $(X,\phi_1) \sim (X,\phi_2)$ if and only if there exists a conformal diffeomorphism $I \colon X_1 \to X_2$ isotopic to $\phi_2 \circ \phi_1^{-1}$. \\

Let $S \subseteq S_g$ be a subset of $n$ points in $S_g$. By the removable singularity theorem, we can think of points in $\mathcal{T}_{g,n}$ as triples $(X,\Sigma,\phi)$, where $X$ is a Riemann surface, $\Sigma \subseteq X$ is a subset of $n$ points in $X$, and $\phi \colon (S_g,S) \to (X,\Sigma)$ is an orientation preserving differomorphism, modulo the equivalence relation $(X_1,\Sigma_1,\phi_1) \sim (X_2,\Sigma_2,\phi_2)$ if and only if there exists a conformal diffeomorphism $I \colon (X_1, \Sigma_1) \to (X_2,\Sigma_2)$ isotopic to $\phi_2 \circ \phi_1^{-1}$ through diffeomorphisms mapping $\Sigma_1$ to $\Sigma_2$.  \\

By the uniformization theorem, $\mathcal{T}_{g,n}$ also parametrizes marked oriented, complete, finite volume hyperbolic structures on $S_{g,n}$ up to isotopy. More precisely, $\mathcal{T}_{g,n}$ is the space of pairs $(X,\phi)$, where $X$ is an oriented, complete, finite volume hyperbolic surface and $\phi \colon S_{g,n} \to X$ is an orientation-preserving diffeomorphism, modulo the equivalence relation $(X,\phi_1) \sim (X,\phi_2)$ if and only if there exists an orientation-preserving isometry $I \colon X_1 \to X_2$ isotopic to $\phi_2 \circ \phi_1^{-1}$. Given  $[(X,\phi)] \in \mathcal{T}_{g,n}$ and an essential simple closed curve $\gamma$ on $S_{g,n}$, we will denote by $\ell_X(\gamma)$ the hyperbolic length of the unique geodesic representative in the free-homotopy class of $\phi(\gamma)$ (the marking $\phi$ is implicit in the notation). \\

We denote the mapping class group of $S_{g,n}$ by $\text{Mod}_{g,n}$. The mapping class group of $S_{g,n}$ acts on $\mathcal{T}_{g,n}$ by change of marking. The quotient $\mathcal{M}_{g,n} := \mathcal{T}_{g,n} / \text{Mod}_{g,n}$ is the moduli space of punctured complex structures on $S_{g,n}$. \\

\textit{The Weil-Peterson volume form.} From the perspective of complex analysis, the Teichm\"ueller space $\mathcal{T}_{g,n}$ can be endowed with a $3g-3+n$ dimensional complex structure. This complex structure admits a natural K\"ahler Hermitian structure. The associated symplectic form $\omega_{wp}$ is called the Weil-Petersson symplectic form. The Weil-Petersson volume form is the top exterior power $v_{wp} := \frac{1}{(3g-3+n)!}\bigwedge^{3g-3+n} \omega_{wp}$. The Weil-Petersson measure is the measure $\mu_{wp}$ induced by the Weil-Petersson volume form on $\mathcal{T}_{g,n}$. See \cite{Hub16} for more details. In \cite{Wol85}, Wolpert obtained the following expression for $\omega_{wp}$ in terms of Fenchel-Nielsen coordinates $(\ell_i,\tau_i)_{i=1}^{3g-3+n} \in (\mathbf{R}_{>0} \times \mathbf{R})^{3g-3+n}$, commonly known as \textit{Wolpert's magic formula}:
\[
\omega_{wp} = \sum_{i=1}^{3g-3+n} d \ell_i \wedge d\tau_i.
\] 
The Weil-Petersson volume form $v_{wp}$ can then be expressed in terms of Fenchel-Nielsen coordinates as
\[
v_{wp} = \prod_{i=1}^{3g-3+n} d \ell_i \wedge d \tau_i.
\]
\vspace{+0.2cm}

\textit{Quadratic differentials.} Let $X$ be a finite type punctured Riemann surface diffeomorphic to $S_{g,n}$. Let $K$ be the canonical bundle of $X$; the holomorphic sections of $K$ are the holomorphic $1$-forms of $X$. A quadratic differential $q$ on $X$ is a holomorphic section of the symmetric square $K \vee K$. Quadratic differentials will be denoted by $(X,q)$, keeping track of the Riemann surface they are defined on. The area of a quadratic differential is $\text{Area}(X,q) := \int_X |q|$. We say $(X,q)$ is integrable if $\text{Area}(X,q) < \infty$. We denote by $QD(X)$ the complex vector space of all integrable, holomorphic quadratic differentials on $X$.  An automorphism of a quadratic differential $(X,q)$ is a conformal diffeomorphism $I \colon X \to X$ such that $I_* q = q$. We denote by $\text{Aut}(X,q)$ the group of automorphisms of $(X,q)$.\\

Alternatively, integrable, holomorphic quadratic differentials $q$ on a Riemann surface $X$ diffeomorphic to $S_{g,n}$ may be interpreted as meromorphic quadratic differentials $q'$, i.e. meromorphic sections of the symmetric square of the canonical bundle, extending $q$ to $X' \supseteq X$, the closed Riemann surface diffeomorphic to $S_g$ obtained by filling in the punctures $\Sigma' \subseteq X'$ of $X$ using the removable singularity theorem, that are holomorphic outside of $\Sigma'$ and have at most simple poles at points of $\Sigma'$. Assume $q'$ is not identically zero. The zeros of $q$ and the punctures $\Sigma' \subseteq X'$  are the \textit{singularities} of $q'$. Such a quadratic differential will be denoted by $(X',\Sigma',q')$, keeping track of the punctures $\Sigma'$, which we interpret as marked points on $X'$. If $q'$ has $m$ unmarked singularities of orders $a_1,\dots,a_m$ and $n$ marked singularities of orders $b_1,\dots,b_n$, then $4g-4 = \sum_{i =1}^m a_i + \sum_{j=1}^n b_j$. \\

Every quadratic differential induces a canonical half-translation structure on the Riemann surface it is defined on. For this reason, we also refer to triples $(X,\Sigma,q)$ as half-translation surfaces. A cylinder $C$ on $(X,\Sigma,q)$ is an isometrically embedded copy in $X$ of an Euclidean cylinder $(\mathbf{R}/c \mathbf{Z}) \times (0,h)$ whose boundary is a union of saddle connections. The number $h$ is the height of the cylinder. The direction of a cylinder is the direction of its boundary saddle connections, considered as an element of $\mathbf{R}\text{P}^1$. The free homotopy class of the closed curve $(\mathbf{R}/c \mathbf{Z}) \times \{\frac{h}{2}\}$ is the core curve of the cylinder. A half-transaltion surface $(X,\Sigma,q)$ is periodic in some direction if $X$ is the union of the cylinders in that direction together with their boundaries.\\

\textit{Square-tiled surfaces.} A square-tiled surface is a half-translation surface that admits a polygon representation $\mathcal{P}$ made up of finitely many unit area squares with sides parallel to the axes of $\mathbf{R}^2 = \mathbf{C}$. The number of squares in $\mathcal{P}$ corresponds to the area of the square-tiled surface. Square-tiled surfaces are both horizontally and vertically periodic. \\

Let $(X,\Sigma,q)$ be a square-tiled surface. Let $C_1,\dots,C_n$ and $D_1 \,\dots,D_m$ be its horizontal and vertical cylinders. Let $h_1, \dots,h_n \in \mathbf{N}$ and $w_1,\dots,w_m \in \mathbf{N}$ be the heights of the horizontal and vertical cylinders. Let $\alpha_1,\dots,\alpha_n$ and $\beta_1,\dots,\beta_m$ be the core curves of the horizontal and vertical cylinders. Consider the integral multicurves $\gamma_1 = \sum_{i =1}^n h_i \cdot \alpha_i$ and $\gamma_2 = \sum_{j=1}^m w_j \cdot \beta_j$ on $X$. We say $\gamma_1$ and $\gamma_2$ are the horizontal and vertical core integral multi-curves of $(X,\Sigma,q)$.\\

\textit{Teichm\"uller and moduli spaces of quadratic differentials.} Let $Q\mathcal{T}_{g,n}$ denote the Teichm\"uller space of marked non-zero, integrable, holomorphic quadratic differentials on $S_{g,n}$. More precisely, $Q\mathcal{T}_{g,n}$ is the set of triples $(X,q,\phi)$, where $X$ is a punctured Riemann surface, $q$ is a non-zero, integrable, holomorphic quadratic differential on $X$, and $\phi \colon S_{g,n} \to X$ is an orientation-preserving diffeomorphism, modulo the equivalence relation $(X_1,q_1,\phi_1) \sim (X_2,q_2,\phi_2)$ if and only if there exists a conformal diffeomorphism $I \colon X_1 \to X_2$ isotopic to $\phi_2 \circ \phi_1^{-1}$ such that $I_* q_1 = q_2$. The forgetful map $Q\mathcal{T}_{g,n} \to \mathcal{T}_{g,n}$ given by $[(X,q,\phi)] \mapsto [(X,\phi)]$ makes $Q\mathcal{T}_{g,n}$ into a bundle over $\mathcal{T}_{g,n}$; after incorporating the zero section, it is actually the complex cotangent bundle of $\mathcal{T}_{g,n}$.\\

Alternatively, let $S \subseteq S_g$ be a subset of $n$ points in $S_g$. Then $Q\mathcal{T}_{g,n}$ may be interpreted as the Teichm\"uller space of marked non-zero, meromorphic quadratic differentials on $S_g$, holomorphic outside of $S$, and having at most simple poles at points of $S$. More precisely, $Q\mathcal{T}_{g,n}$ is the set of tuples $(X,\Sigma,q,\phi)$, where $X$ is a punctured Riemann surface, $\Sigma \subseteq X$ is a subset of $n$ points in $X$, $q$ is a non-zero, meromorphic quadratic differential on $X$, holomorphic outside of $\Sigma$, and with at most simple poles at points of $\Sigma$, and $\phi \colon (S_{g,n},S) \to (X,\Sigma)$ is an orientation-preserving diffeomorphism, modulo the equivalence relation $(X_1,\Sigma_1,q_1,\phi_1) \sim (X_2,\Sigma_2,q_2,\phi_2)$ if and only if there exists a conformal diffeomorphism $I \colon (X_1,\Sigma_1) \to (X_2,\Sigma_2)$ isotopic to $\phi_2 \circ \phi_1^{-1}$ through diffeomorphisms mapping $\Sigma_1$ to $\Sigma_2$ and such that $I_* q_1 = q_2$. We usually refer to $Q\mathcal{T}_{g,n}$ simply as the Teichm\"uller space of marked quadratic differentials on $S_{g,n}$. \\

The space $Q\mathcal{T}_{g,n}$ has a natural stratification induced by the order of singularities. Each connected components of every stratum of $Q\mathcal{T}_{g,n}$ has a natural complex structure induced by period coordinates. Marked square-tiled surfaces in $Q\mathcal{T}_{g,n}$ will be denoted by $Q\mathcal{T}_{g,n}(\mathbf{Z})$, as they correspond to integer points in period coordinates; we also refer to them as the integer points of $Q\mathcal{T}_{g,n}$.\\

The mapping class group $\text{Mod}_{g,n}$ acts on $Q\mathcal{T}_{g,n}$ by change of marking. The quotient $Q\mathcal{M}_{g,n} := Q\mathcal{T}_{g,n}/\text{Mod}_{g,n}$ is the moduli space of non-zero, integrable, holomorphic quadratic differentials on $S_{g,n}$. We usually refer to $Q\mathcal{M}_{g,n}$ simply as the moduli space of quadratic differentials on $S_{g,n}$. Square-tiled surfaces in $Q\mathcal{M}_{g,n}$ will be denoted by $Q\mathcal{M}_{g,n}(\mathbf{Z})$. Given any $[(X,q)] \in Q\mathcal{M}_{g,n}$, there is a natural bijection between $\text{Aut}(X,q)$ and the $\text{Mod}_{g,n}$-stabilizer of any point $[(X,q,\phi)]$ covering $[(X,q)]$ under the quotient map $Q\mathcal{T}_{g,n} \to Q\mathcal{M}_{g,n}$.\\

\textit{Measured geodesic laminations and singular measured foliations.} A geodesic lamination $\lambda$ on a complete, finite volume hyperbolic surface $X$ diffeomorphic to $S_{g,n}$ is a set of disjoint simple, complete geodesics whose union is a compact subset of $X$. A measured geodesic lamination is a geodesic lamination carrying an invariant transverse measure fully supported on the lamination. We can understand measured geodesic laminations by lifting them to a universal cover $\mathbf{H}^2 \to X$. A non-oriented geodesic on $\mathbf{H}^2$ is specified by a set of distinct points on the boundary at infinity $\partial^\infty \mathbf{H}^2 = S^1$. It follows that measured geodesic laminations on  diffeomorphic hyperbolic surfaces may be compared by passing to the boundary at infinity of their universal covers. Thus, the space of measured geodesic laminations on $X$ depends only on the underlying topological surface $S_{g,n}$. We denote the space of measured geodesic laminations on $S_{g,n}$ by $\mathcal{ML}_{g,n}$. It can be topologized by embedding it into the space of geodesic currents on $S_{g,n}$. By taking  geodesic representatives, integral multi-curves on $S_{g,n}$ can be interpreted as elements of $\mathcal{ML}_{g,n}$; we denote them by $\mathcal{ML}_{g,n}(\mathbf{Z})$ and refer to them as the integer points of $\mathcal{ML}_{g,n}$. For more details on the theory of measured geodesic laminations see \cite{Bon88}, \cite{Bon97}, and 8.3 in \cite{Mar16}.\\

Given two essential simple closed curves $\gamma_1$ and $\gamma_2$ on $S_{g,n}$, their geometric intersection number $i(\gamma_1,\gamma_2)$ is the minimum number of intersections among all transverse free-homotopy representatives of $\gamma_1$ and $\gamma_2$. Geometric intersection number can be extended by homogeneity and continuity to a pairing on $\mathcal{ML}_{g,n}$, which we still refer to as geometric intersection number. More precisely, there is a unique continuous, symmetric, bilinear form $i \colon \mathcal{ML}_{g,n} \times \mathcal{ML}_{g,n} \to \mathbf{R}_{\geq 0}$ which agrees with geometric intersection number on essential simple closed curves. For a proof see \cite{Bon88}.\\

We denote by $\mathcal{MF}_{g,n}$ the space of singular measured foliations on $S_{g,n}$ up to isotopy and Whitehead moves. There is a natural correspondence between singular measured foliations and measured geodesic laminations. Indeed, one can define a homeomorphism $\mathcal{MF}_{g,n}\to \mathcal{ML}_{g,n}$ by straightening the leaves of singular measured foliations to obtain measured geodesic laminations. In particular, it makes sense to talk about intersection numbers and integer points $\mathcal{MF}_{g,n}(\mathbf{Z}) \subseteq \mathcal{MF}_{g,n}$ of singular measured foliations. For more details on the theory of singular measured foliations see 11.2.2 in \cite{FM11} and Expos\'e 5 in \cite{FLP12}. For more details on the correspondence between singular measured foliations and measured geodesic laminations see \cite{Lev83}.\\

\textit{The Thurston measure.}  The space of measured geodesic laminations $\mathcal{ML}_{g,n}$ admits a $6g-6+2n$ dimensional piecewise integral linear structure induced by train track charts. The integer points of this structure are precisely the integral multi-curves $\mathcal{ML}_{g,n}(\mathbf{Z}) \subseteq \mathcal{ML}_{g,n}$. For each $L > 0$, consider the counting measure $\mu^L$ on $\mathcal{ML}_{g,n}$ given by
\begin{equation}
\label{ML_counting_measure}
\mu^L := \frac{1}{L^{6g-6+2n}} \sum_{\gamma \in \mathcal{ML}_{g,n}(\mathbf{Z})} \delta_{ \frac{1}{L} \cdot \gamma}.
\end{equation}
As $L \to \infty$, this sequence of counting measures converges to a non-zero, locally finite measure $\mu_{\text{Thu}}$ on $\mathcal{ML}_{g,n}$ called the Thurston measure. This measure is $\text{Mod}_{g,n}$-invariant and belongs to the Lebesgue measure class. It also satisfies the following scaling property: $\mu_{\text{Thu}}(t \cdot A) = t^{6g-6+2n} \cdot \mu_{\text{Thu}}(A)$ for every measurable set $A \subseteq \mathcal{ML}_{g,n}$ and every $t > 0$.\\

Train track charts also induce a $\text{Mod}_{g,n}$-invariant symplectic form $\omega_{\text{Thu}}$ on $\mathcal{ML}_{g,n}$ called the Thurston symplectic form. For more details on the definition of $\omega_{\text{Thu}}$ see $\S 3.2$ in \cite{PH92}. The top exterior power $v_{\text{Thu}} := \frac{1}{(3g-3+n)!} \bigwedge^{3g-3+n}\omega_{\text{Thu}}$ is called the Thurston volume form. In \cite{Mas85}, Masur showed that the action of $\text{Mod}_{g,n}$ on $\mathcal{ML}_{g,n}$ is ergodic with respect to $\mu_{\text{Thu}}$. As a consequence, $\mu_{\text{Thu}}$ is the unique, up to scaling, $\text{Mod}_{g,n}$-invariant measure on $\mathcal{ML}_{g,n}$ in the Lebesgue measure class. It follows that the measure induced by the Thurston volume form on $\mathcal{ML}_{g,n}$ is a multiple of $\mu_{\text{Thu}}$. Moreover, the scaling factor relating these measures can be computed explicitely.\\

\begin{proposition}
	\label{Thurston_measure_scaling_factor}
	If $\nu_{\text{Thu}}$ denotes the measure induced by the Thurston volume form on $\mathcal{ML}_{g,n}$, then
	\[
	\nu_{\text{Thu}} = 2^{2g-3+n} \cdot \mu_{Thu}.
	\]
\end{proposition}
$ $\vspace{+0.2cm}

\textit{The Hubbard-Masur map.} Let $\text{Re}, \text{Im} \colon Q\mathcal{T}_{g,n} \to \mathcal{ML}_{g,n}$ be the maps that assign to every marked non-zero, integrable, holomorphic quadratic differential in $Q\mathcal{T}_{g,n}$ its real and imaginary foliations interpreted as element of $\mathcal{ML}_{g,n}$. These maps are $\text{Mod}_{g,n}$-equivariant, so they induce maps $\textbf{Re}, \textbf{Im} \colon Q\mathcal{M}_{g,n} \to \mathcal{ML}_{g,n} / \text{Mod}_{g,n}$ on quotients. Let $\Delta \subseteq \mathcal{ML}_{g,n} \times \mathcal{ML}_{g,n}$ be the closed subset
\[
\Delta := \{(\lambda, \mu) \in \mathcal{ML}_{g,n} \times \mathcal{ML}_{g,n} \ | \ \exists  \alpha \in \mathcal{ML}_{g,n}: \  i(\lambda, \alpha) + i (\mu,\alpha) = 0\}.
\]
The Hubbard-Masur map introduced in (\ref{hubbard_masur_definition}) is then given by
\[
\begin{array}{c c c c}
h: & Q\mathcal{T}_{g,n} &\to & \mathcal{ML}_{g,n} \times \mathcal{ML}_{g,n} - \Delta \\
& [(X,q,\phi)] &\mapsto & (\text{Re}([(X,q,\phi)]), \text{Im}([(X,q,\phi)]))
\end{array}.
\]
$ $\vspace{+0.2cm}

\textit{Counting measures on $\text{Re}^{-1}(\lambda)$.} The subset $\text{Re}^{-1}(\lambda) \subseteq Q\mathcal{T}_{g,n}$ inherits a stratification from the one on $Q\mathcal{T}_{g,n}$. Relevant coordinates on connected components of stata of $\text{Re}^{-1}(\lambda)$ are obtained by considering the real parts of period coordinates on the corresponding connected components of strata of $Q\mathcal{T}_{g,n}$; we refer to such coordinates as \textit{real period coordinates}. For every $L>0$, consider the counting measure $m_{\text{Re}^{-1}(\lambda)}^L$ on $\text{Re}^{-1}(\lambda)$ given by
\begin{equation}
\label{counting_measure_re}
m_{\text{Re}^{-1}(\lambda)}^L := \frac{1}{L^{6g-6+2n}} \sum_{ \substack{[(X,q,\phi)] \in \text{Re}^{-1}(\lambda), \\ \text{Im}([(X,q,\phi)]) \in \frac{1}{L} \cdot \mathcal{ML}_{g,n}(\mathbf{Z})}} \delta_{[(X,q,\phi)]}.
\end{equation}
As $L \to \infty$, this sequence of counting measures converges to a non-zero, locally finite measure $m_{\text{Re}^{-1}(\lambda)}$ on $\text{Re}^{-1}(\lambda)$. This measure is $\text{Stab}(\lambda)$-invariant and belongs to the Lebesgue measure class. The principal stratum of $\text{Re}^{-1}(\lambda)$ is a full measure, open subset of $\text{Re}^{-1}(\lambda)$.\\

\textit{Measure theory of properly discontinuous group actions.} Let $X$ be a locally compact, Hausdorff, second countable topological space endowed with a properly discontinuous action of a group $G$. Notice $X/G$ is also a locally compact, Hausdorff, second countable topological space. Let $\pi \colon X \to X/G$ be the associated quotient map. As the action of $G$ on $X$ is properly discontinuous, we can cover $X$ by open sets $U$ invariant under the action of finite subgroups $\Gamma < G$ such that $gU \cap U = \emptyset$ for all $g \in G - \Gamma$. The quotient map restricts to $\pi|_U \colon U \to U/\Gamma \subseteq X/G$. Open sets $U/\Gamma \subseteq X/G$ of this form will be refered to as \textit{well covered}. Given a locally finite, $G$-invariant Borel measure $\mu$ on $X$, there is a canonical local pushforward measure $\pi_* \mu$ on $X/G$ defined in the following way.\\

\begin{definition} 
	\label{local_pushforward_measure}
	The local pushforward $\pi_* \mu$ is the unique locally finite Borel measure on $X/G$ satisfying the following property: If $U/\Gamma \subseteq X/G$ is a well covered open set, then $(\pi_* \mu)|_{U/\Gamma} =  \frac{1}{\# \Gamma} \cdot (\pi|_U)_\# (\mu|_U)$, where $(\pi|_U)_\# (\mu|_U)$ denotes the usual pushforward of the measure $\mu|_U$ under the map $\pi|_U$.\\
\end{definition}

Restriction to open sets and taking local pushforward are continuous operations between spaces of locally finite Borel measures endowed with the weak topology.\\

Suppose that $\mu$ is a locally finite, $G$-invariant Borel measure on $X$ of the form
\[
\mu = \sum_{x \in A} w(x) \cdot \delta_x,
\]
where $A \subseteq X$ is a $G$-invariant, discrete, closed subset of $X$ and $w \colon A \to \mathbf{R}_{\geq 0}$ is a $G$-invariant function. Then the local pushforward $\pi_* \mu$ is given by
\begin{equation}
\label{local_pushforward_atomic}
\pi_* \mu = \sum_{[x] \in A/G} \frac{1}{\# \text{Stab}(x)} \cdot w(x) \cdot \delta_{[x]}.
\end{equation}
\vspace{+0.2cm}

\textit{Growth of the number of simple closed hyperbolic geodesics.} We briefly review Mirzakhani's curve counting results in \cite{Mir08b}. Given a hyperbolic surface $X \in \mathcal{M}_{g,n}$ and a rational multi-curve $\gamma$ on $X$, consider the counting function
\begin{equation}
\label{counting_curves}
s(X,\gamma,L) := \# \{\alpha \in \text{Mod}_{g,n} \cdot \gamma \ | \ \ell_X(\alpha) \leq L \}.
\end{equation}
In words, $s(X,\gamma,L)$ is the number rational multi-curves on $X$ of topological type $[\gamma]$ and hyperbolic length $\leq L$. The following result, which is Theorem 1.1 in \cite{Mir08b}, describes the growth of the counting function $s(X,\gamma,L)$ as $L \to \infty$.\\

\begin{theorem}
	\label{mirzakhani_growth}
	For any hyperbolic surface $X \in \mathcal{M}_{g,n}$ and any rational muti-curve $\gamma$ on $X$,
	\[
	\lim_{L \to \infty} \frac{s(X,\gamma,L)}{L^{6g-6+2n}} = n_\gamma(X),
	\]
	where $n_\gamma \colon \mathcal{M}_{g,n} \to \mathbf{R}_{>0}$ is a continuous, proper function.\\
\end{theorem}

For every $X \in \mathcal{T}_{g,n}$, consider the compact subset $B_X \subseteq \mathcal{ML}_{g,n}$ given by
\[
B_X := \{\lambda \in \mathcal{ML}_{g,n} \ | \ \ell_X(\lambda) \leq 1 \}.
\]
Consider the continuous map
\[
\begin{array}{c c c c}
B: & T_{g,n} &\to & \mathbf{R}_{> 0} \\
& X &\mapsto & \mu_{\text{Thu}}(B_X)
\end{array}.
\]
Let $\widehat{\mu}_{wp}$ be the local pushforward of the Weil-Petersson measure $\mu_{wp}$ on $\mathcal{T}_{g,n}$ under the quotient map $\mathcal{T}_{g,n} \to \mathcal{M}_{g,n}$. In \cite{Mir08b}, the following integral, relevant in the statements of theorems that follow, is proved to be finite, and moreover, a positive rational multiple of $\pi^{6g-6+2n}$:
\[
b_{g,n} := \int_{\mathcal{M}_{g,n}} B(X) \ d\widehat{\mu}_{wp}.
\]
\vspace{+0.2cm}

To every rational multi-curve $\gamma$ on $S_{g,n}$ one can associate a positive rational number $c(\gamma) \in \mathbf{Q}_{>0}$ in the following way. Consider the integral
\[
P(L,\gamma) := \int_{\mathcal{M}_{g,n}} s(X,\gamma,L) \ d\widehat{\mu}_{wp}.
\]
Mirzakhani's Weil-Petersson integration techniques in \cite{Mir07} are applied in \linebreak \cite{Mir08b} to show that $P(L,\gamma)$ is a polynomial of degree $6g - 6 + 2n$, closely related to the Weil-Petersson volume polynomial of the moduli space of bordered Riemann surfaces homeomorphic to the surface with boundary obtained by cutting $S$ along $\gamma$. Let $c(\gamma)$ be the leading coefficient of this polynomial, that is
\[
c(\gamma) := \lim_{L \to \infty} \frac{P(L,\gamma)}{L^{6g-6+2n}}.
\]
Explicit formulas for computing $c(\gamma)$ as a sum of Weil-Petersson volumes are given in \cite{Mir08b}. In particular, it is proved that $c(\gamma) \in \mathbf{Q}_{>0}$. Notice that $c(\gamma)$ depends only on the topological type of $\gamma$. We will refer to $c(\gamma)$ as the \textit{frequency} of multi-curves on $S_{g,n}$ of topological type $[\gamma]$. \\

The following result, which is Theorem 1.2 in \cite{Mir08b}, describes the dependence of the function $n_\gamma \colon \mathcal{M}_{g,n} \to \mathbf{R}_{>0}$ in Theorem \ref{mirzakhani_growth} with respect to the hyperbolic structure $X \in \mathcal{M}_{g,n}$ and the rational multi-curve $\gamma$ on $S_{g,n}$; it is a direct consequence of Theorem \ref{mirzakhani_measure_convergence_intro}.\\

\begin{theorem}
	\label{mirzakhani_function}
	For any hyperbolic surface $X \in \mathcal{M}_{g,n}$ and any rational muti-curve $\gamma$ on $X$,
	\[
	n_{\gamma}(X) = \frac{c(\gamma) \cdot B(X)}{b_{g,n}}.
	\]
\end{theorem}
\vspace{+0.2cm}

\section{Proofs of main results}

\textit{Proof of Theorem \ref{counting_quadratic_differentials_intro}.} We now present the proof of Theorem \ref{counting_quadratic_differentials_intro} in full detail. Let $\gamma_1$ and $\gamma_2$ be two integral multi-curves on $S_{g,n}$. Recall the definition of the function  $s(\gamma_1,\gamma_2,L)$ in (\ref{counting_function_square_tiled_surfaces}). It is convenient to write $s(\gamma_1,\gamma_2,L)$ as
\[
s(\gamma_1,\gamma_2,L) = \sum_{\substack{[[(X,q,\phi)]] \in Q\mathcal{T}_{g,n}/\text{Mod}_{g,n}, \\ \text{Re}([(X,q,\phi)]) \in \text{Mod}_{g,n} \cdot \gamma_1, \ \text{Im}([(X,q,\phi)]) \in  \text{Mod}_{g,n} \cdot \gamma_2, \\ \text{Area}([(X,q,\phi)]) \leq L }} \frac{1}{\# \text{Stab}([(X,q,\phi)])},
\]
where $\text{Stab}([(X,q,\phi)])$ is the stabilizer of the marked quadratic differential $[(X,q,\phi)] \allowbreak \in Q\mathcal{T}_{g,n}$ with respect to the action of $\text{Mod}_{g,n}$ on $Q\mathcal{T}_{g,n}$. To better understand $s(\gamma_1,\gamma_2,L)$ we relate it to a counting problem on $\mathcal{ML}_{g,n}$ by using the Hubbard-Masur map. It follows from Theorem \ref{hubbard_masur_map} that
\[
s(\gamma_1,\gamma_2,L) = \sum_{\substack{[(\alpha,\beta)] \in (\mathcal{ML}_{g,n}\times \mathcal{ML}_{g,n}- \Delta)/\text{Mod}_{g,n}, \\ \alpha \in \text{Mod}_{g,n} \cdot \gamma_1,  \ \beta \in \text{Mod}_{g,n} \cdot \gamma_2, \\ i(\alpha,\beta) \leq L }} \frac{1}{\# (\text{Stab}(\alpha) \cap \text{Stab}(\beta))},
\]
where $\text{Stab}(\alpha)$ and $\text{Stab}(\beta)$ are the stabilizers of the multi-curves $\alpha$ and $\beta$ with respect to the $\text{Mod}_{g,n}$ action on $\mathcal{ML}_{g,n}$. In this sum it is enough to consider equivalence classes of pairs of the form $(\gamma_1,\beta) \in \mathcal{ML}_{g,n}\times \mathcal{ML}_{g,n}- \Delta$ with $\beta \in \text{Mod}_{g,n} \cdot \gamma_2$ and $i(\gamma_1,\beta) \leq L$. Notice that $[\phi] \in \text{Mod}_{g,n}$ identifies the pairs $(\gamma_1,\beta)$ and $(\gamma_1,\beta')$ if and only if $[\phi].\beta = \beta'$ and $\phi \in \text{Stab}(\gamma_1)$. In particular, it is enough to consider the action of $\text{Stab}(\gamma_1)$ instead of the action of the whole mapping class group $\text{Mod}_{g,n}$. Recall that by definition $\beta \in \mathcal{ML}_{g,n}(\gamma_1)$ if and only if $(\gamma_1,\beta) \in \mathcal{ML}_{g,n}\times \mathcal{ML}_{g,n}- \Delta$. With all these considerations in mind we get
\begin{equation}
\label{simplified_sum_2}
s(\gamma_1,\gamma_2,L) = \sum_{\substack{[\beta]\in \mathcal{ML}_{g,n}(\gamma_1)/\text{Stab}(\gamma_1), \\ \beta \in \text{Mod}_{g,n} \cdot \gamma_2, \\ i(\gamma_1,\beta) \leq L }} \frac{1}{\# (\text{Stab}(\gamma_1) \cap \text{Stab}(\beta))}.
\end{equation}
\vspace{+0.2cm}

Once in this setting we can use Mirzakhani's curve counting techniques as follows. Recall the definition of the counting measures $\mu_{\gamma_2}^L$ on $\mathcal{ML}_{g,n}$ in (\ref{curve_counting_measures}). Theorem \ref{mirzakhani_measure_convergence_intro} shows that $\mu_{\gamma_2}^L \to \frac{c(\gamma_2)}{b_{g,n}}\cdot \mu_{\text{Thu}}$ as $L \to \infty$. As $\mathcal{ML}_{g,n}(\gamma_1) \subseteq \mathcal{ML}_{g,n}$ is open,  $\mu_{\gamma_1}^L|_{\mathcal{ML}_{g,n}(\gamma_1)} \to \frac{c(\gamma_1)}{b_g} \cdot \mu_{\text{Thu}}|_{\mathcal{ML}_{g,n}(\gamma_1)}$ as $L \to \infty$. The following lemma, which is a direct consequence of Theorem \ref{hubbard_masur_slice}, is crucial to pass to the quotient $\mathcal{ML}_{g,n}(\gamma_1)/\text{Stab}(\gamma_1)$.\\

\begin{lemma}
	The action of $\text{Stab}(\gamma_1)$ on $\mathcal{ML}_{g,n}(\gamma_1)$ is properly discontinuous.\\
\end{lemma}

\begin{proof}
	Theorem \ref{hubbard_masur_slice} provides a $\text{Stab}(\gamma_1)$-equivariant homeomorphism \linebreak $\mathcal{ML}_{g,n}(\gamma_1) \to \mathcal{T}_{g,n}$. The group $\text{Stab}(\gamma_1)$ acts properly discontinously on $\mathcal{T}_{g,n}$ because the whole mapping class group $\text{Mod}_{g,n}$ does so. We deduce that the action of $\text{Stab}(\gamma_1)$ on $\mathcal{ML}_{g,n}(\gamma_1)$ is properly discontinuous.
\end{proof}
$ $\vspace{+0.15cm}

The measures $\mu_{\gamma_2}^L|_{\mathcal{ML}_{g,n}(\gamma_1)}$ and $\mu_{\text{Thu}}|_{\mathcal{ML}_{g,n}(\gamma_1)}$ on $\mathcal{ML}_{g,n}(\gamma_1)$ are $\text{Stab}(\gamma_1)$-invariant because they are the restriction of $\text{Mod}_{g,n}$-invariant measures to a $\text{Stab}(\gamma_1)$-invariant set. Following Definition \ref{local_pushforward_measure} we construct the local pushforwards $\widehat{\mu}_{\gamma_2}^L$ and $\widehat{\mu}_{\text{Thu}}$ of these measures  under the quotient map $\mathcal{ML}_{g,n}(\gamma_1) \to \mathcal{ML}_{g,n}(\gamma_1)/\text{Stab}(\gamma_1)$. As taking local pushforward is a continuous operation, we deduce the following:\\

\begin{proposition}
	\label{pushforward_count_measure_convergence}
	For any pair of integral multi-curves $\gamma_1$ and $\gamma_2$ on $S_{g,n}$,
	\[
	\widehat{\mu}_{\gamma_2}^L\to \frac{c(\gamma_2)}{b_{g,n}} \cdot \widehat{\mu}_{\text{Thu}}.
	\]
\end{proposition}
\vspace{+0.25cm}

Consider the subsets
\begin{align*}
B(\gamma_1) &:= \{\lambda \in \mathcal{ML}_{g,n}(\gamma_1) \ | \ i(\gamma_1,\lambda) \leq 1\} \subseteq \mathcal{ML}_{g,n}(\gamma_1), \\
\widehat{B}(\gamma_1) &:= B(\gamma_1)/\text{Stab}(\gamma_1) \subseteq \mathcal{ML}_{g,n}(\gamma_1)/\text{Stab}(\gamma_1).
\end{align*}
The following proposition brings us back to our original counting problem.\\

\begin{proposition}
	\label{measure_relates_to_count}
	For any pair of integral multi-curves $\gamma_1$ and $\gamma_2$ on $S_{g,n}$ and any $L >0$,
	\[
	\frac{s(\gamma_1,\gamma_2,L)}{L^{6g-6+2n}} = \widehat {\mu}_{\gamma_2}^L(\widehat{B}(\gamma_1)).
	\]
\end{proposition}
\vspace{+0.2cm}

\begin{proof}
	Following (\ref{local_pushforward_atomic}) and (\ref{simplified_sum_2}) we have
	\begin{align*}
	\widehat {\mu}_{\gamma_2}^L(\widehat{B}(\gamma_1)) &= \frac{1}{L^{6g-6+2n}}  \sum_{\substack{[\beta]\in \mathcal{ML}_{g,n}(\gamma_1)/\text{Stab}(\gamma_1), \\ \beta \in \text{Mod}_{g,n} \cdot \gamma_2}} \frac{1}{\# (\text{Stab}(\gamma_1) \cap \text{Stab}(\beta))} \cdot \delta_{[\beta]}(\widehat{B}(\gamma_1))\\
	&= \frac{1}{L^{6g-6+2n}} \sum_{\substack{[\beta]\in \mathcal{ML}_{g,n}(\gamma_1)/\text{Stab}(\gamma_1), \\ \beta \in \text{Mod}_{g,n} \cdot \gamma_2, \\ i(\gamma_1,\beta) \leq L }}  \frac{1}{\# (\text{Stab}(\gamma_1) \cap \text{Stab}(\beta))}\\
	&= \frac{s(\gamma_1,\gamma_2,L)}{L^{6g-6+2n}}. \qedhere
	\end{align*}
\end{proof}
$ $\vspace{+0.15cm}

From Proposition \ref{measure_relates_to_count} it follows that to prove Theorem \ref{counting_quadratic_differentials_intro} it is enough to prove the following result.\\

\begin{proposition}
	\label{almost_main_theorem}
	For any pair of integral multi-curves $\gamma_1$ and $\gamma_2$ on $S_{g,n}$,
	\[
	\widehat{\mu}_{\gamma_2}^L(\widehat{B}(\gamma_1)) \to \frac{c(\gamma_2)}{b_{g,n}} \cdot \widehat{\mu}_{\text{Thu}}(\widehat{B}(\gamma_1)).
	\]
\end{proposition}
$ $\vspace{+0.2cm}

Let us finish the proof of Theorem \ref{counting_quadratic_differentials_intro} assuming Proposition \ref{almost_main_theorem} is true.\\

\begin{proof}[Proof of Theorem \ref{counting_quadratic_differentials_intro}]
	It follows from Propositions \ref{measure_relates_to_count} and \ref{almost_main_theorem} that
	\begin{equation}
	\label{almost_finished_proof}
	\frac{s(\gamma_1,\gamma_2,L)}{L^{6g-6+2n}} \to \frac{c(\gamma_2)}{b_{g,n}} \cdot \widehat{\mu}_{\text{Thu}}(\widehat{B}(\gamma_1))
	\end{equation}
	as $L \to \infty$. As $s(\gamma_1,\gamma_2,L) = s(\gamma_2,\gamma_1,L)$, we also have that
	\[
	\frac{s(\gamma_1,\gamma_2,L)}{L^{6g-6+2n}} \to \frac{c(\gamma_1)}{b_{g,n}} \cdot \widehat{\mu}_{\text{Thu}}(\widehat{B}(\gamma_2))
	\] 
	as $L \to \infty$. We deduce
	\[
	\frac{\widehat{\mu}_{\text{Thu}}(\widehat{B}(\gamma_1))}{c(\gamma_1)} = \frac{\widehat{\mu}_{\text{Thu}}(\widehat{B}(\gamma_2))}{c(\gamma_2)}.
	\] 
	As this holds for all integral multi-curves $\gamma_1$ and $\gamma_2$ on $S_{g,n}$, it follows that $r_{g,n} := \frac{\widehat{\mu}_{\text{Thu}}(\widehat{B}(\gamma))}{ c(\gamma)}$ is a constant depending only on $g$ and $n$ and not on the integral multi-curve $\gamma$. Explicit computations when $\gamma$ is a pair of pants decomposition of $S_{g,n}$ show that $r_{g,n} = \frac{1}{2^{2g-3+n}}$. The details of such computations are presented in Section 4. Theorem \ref{counting_quadratic_differentials_intro} then follows from (\ref{almost_finished_proof}).
\end{proof}
$ $ \vspace{+0.15cm}

It remains to prove Proposition \ref{almost_main_theorem}. Using the scaling properties of the Thurston measure one can check that $\widehat{\mu}_{\text{Thu}}(\partial \widehat{B}(\gamma_1)) = 0$. Yet we cannot directly apply Portmanteau's Theorem, see for instance Proposition 1.3.8 in \cite{Mar16}, to conclude from Proposition \ref{pushforward_count_measure_convergence} that $\widehat{\mu}_{\gamma_2}^L(\widehat{B}(\gamma_1)) \to \frac{c(\gamma_2)}{b_{g,n}} \cdot \widehat{\mu}_{\text{Thu}}(\widehat{B}(\gamma_1))$ as $L \to \infty$ because $\widehat{B}(\gamma_1) \subseteq \mathcal{ML}_{g,n}(\gamma_1)/ \text{Stab}(\gamma_1)$ is not compact. To prove such convergence we show 
the following no escape of mass property holds.\\

\begin{proposition}
	\label{no_escape_of_mass}
	For every $\epsilon > 0$ there exists a compact subset $K_\epsilon \subseteq \widehat{B}(\gamma_1)$ with the following properties:
	\begin{enumerate}
		\item $\frac{c(\gamma_2)}{b_{g,n}} \cdot \widehat{\mu}_{\text{Thu}}(\partial K_\epsilon) = 0$.
		\item $\frac{c(\gamma_2)}{b_{g,n}} \cdot \widehat{\mu}_{\text{Thu}}(\widehat{B}(\gamma_1)\backslash K_\epsilon) < \epsilon$.
		\item $\widehat{\mu}_{\gamma_2}^L(\widehat{B}(\gamma_1) \backslash K_\epsilon) < \epsilon$ for all big enough $L > 0$.\\
	\end{enumerate}
\end{proposition}

We refer to the situation described in Proposition \ref{no_escape_of_mass} as there being \textit{no escape of mass} in $\widehat{B}(\gamma_1)$ for the measures $\widehat{\mu}_{\gamma_2}^L \to \frac{c(\gamma_2)}{b_{g,n}} \cdot \widehat{\mu}_{\text{Thu}}$. Let us finish the proof of Proposition \ref{almost_main_theorem} assuming Proposition \ref{no_escape_of_mass} is true.\\

\begin{proof}[Proof of Proposition \ref{almost_main_theorem}]
	Fix $\epsilon > 0$. Let $K_\epsilon \subseteq \widehat{B}(\gamma_1)$ be a compact subset as in Proposition \ref{no_escape_of_mass}. As $K_\epsilon \subseteq \mathcal{ML}_{g,n}(\gamma_1)/\text{Stab}(\gamma_1)$ is compact with $\widehat{\mu}_{\text{Thu}}(\partial K_\epsilon) = 0$, Proposition \ref{pushforward_count_measure_convergence} and Portmanteau's theorem imply $\widehat{\mu}_{\gamma_2}^L(K_\epsilon) \to \frac{c(\gamma_2)}{b_{g,n}} \cdot \widehat{\mu}_{\text{Thu}}(K_\epsilon)$ as $L \to \infty$. Let $L > 0$ be big enough so that $|\frac{c(\gamma_2)}{b_{g,n}} \cdot \widehat{\mu}_{\text{Thu}}(K_\epsilon) - \widehat{\mu}_{\gamma_2}^L(K_\epsilon)| < \epsilon$ and $\widehat{\mu}_{\gamma_2}^L(\widehat{B}(\gamma_1) \backslash K_\epsilon) < \epsilon$. The triangle inequality yields
	\[
	\bigg\vert \textstyle\frac{c(\gamma_2)}{b_{g,n}} \cdot \widehat{\mu}_{\text{Thu}}(\widehat{B}(\gamma_1)) - \widehat{\mu}_{\gamma_2}^L(\widehat{B}(\gamma_1))\bigg\vert \leq 3 \cdot \epsilon.
	\]
	As $\epsilon > 0$ was arbitrary, this finishes the proof.
\end{proof}
$ $ \vspace{+0.15cm}

\textit{No escape of mass.} To prove Proposition \ref{no_escape_of_mass} we move back into the realm of quadratic differentials, where we use period coordinates to reduce our problem to a lattice counting argument in Euclidean space. Let us first reduce to a no escape of mass problem for simpler measures. Recall the definition of the counting measures $\mu^L$ on $\mathcal{ML}_{g,n}$ in (\ref{ML_counting_measure}). By definition $\mu^L \to \mu_{\text{Thu}}$ as $L \to \infty$. As restriction to open sets is a continuous operation, $\mu^L|_{\mathcal{ML}_{g,n}(\gamma_1)} \to \mu_{\text{Thu}}|_{\mathcal{ML}_{g,n}(\gamma_1)}$ as $L \to \infty$. Let $\widehat{\mu}^L$ be the local pushforward of the measure $\mu^L|_{\mathcal{ML}_{g,n}(\gamma_1)}$ under the quotient map $\mathcal{ML}_{g,n}(\gamma_1) \to \mathcal{ML}_{g,n}(\gamma_1)/\text{Stab}(\gamma_1)$. As taking local pushforward is a continuous operation, $\widehat{\mu}^L \to \widehat{\mu}_{\text{Thu}}$ as $L \to \infty$. As $\widehat{\mu}_{\gamma_2}^L \leq \widehat{\mu}^L$, it is enough for our purposes to prove there is no escape of mass in $\widehat{B}(\gamma_1)$ for the measures $\widehat{\mu}^L \to \widehat{\mu}_{\text{Thu}}$. \\

Recall the definition of the counting measures $m_{\text{Re}^{-1}(\gamma_1)}^L$ on $\text{Re}^{-1}(\gamma_1)$ in (\ref{counting_measure_re}). By definition $m_{\text{Re}^{-1}(\gamma_1)}^L \to m_{\text{Re}^{-1}(\gamma_1)}$ as $L \to \infty$. Let $\widehat{m}^L_{\text{Re}^{-1}(\gamma_1)}$ and $\widehat{m}_{\text{Re}^{-1}(\gamma_1)}$ be the local pushforwards of these measures under the quotient map $\text{Re}^{-1}(\gamma_1) \to \text{Re}^{-1}(\gamma_1)/\text{Stab}(\gamma_1)$. As taking local pushforward is a continuous operation, \linebreak $\widehat{m}^L_{\text{Re}^{-1}(\gamma_1)} \to \widehat{m}_{\text{Re}^{-1}(\gamma_1)}$ as $L \to \infty$. Consider the subsets
\begin{align*}
D(\gamma_1) &:= \{[(X,q,\phi)] \in \text{Re}^{-1}(\gamma_1)\ | \ \text{Area}([(X,q,\phi)]) \leq 1\} \subseteq \text{Re}^{-1}(\gamma_1), \\
\widehat{D}(\gamma_1) &:= D(\gamma_1)/\text{Stab}(\gamma_1) \subseteq \text{Re}^{-1}(\gamma_1)/\text{Stab}(\gamma_1).
\end{align*}
By Theorem \ref{hubbard_masur_map}, the inverse of the Hubbard-Masur map induces a $\text{Stab}(\gamma_1)$-\linebreak equivariant homeomorphism $\mathcal{ML}_{g,n}(\gamma_1) \to \text{Re}^{-1}(\gamma_1)$ sending geometric intersection number with $\gamma_1$ to area of quadratic differentials. In particular, such map sends $B(\gamma_1)$ to $D(\gamma_1)$. Moreover, by definition, such map sends $\mu^L|_{\mathcal{ML}_{g,n}(\gamma_1)}$ to \linebreak $m_{\text{Re}^{-1}(\gamma_1)}^L$ and  $\mu_{\text{Thu}}|_{\mathcal{ML}_{g,n}(\gamma_1)}$ to $m_{\text{Re}^{-1}(\gamma_1)}$. Our problem then translates to showing there is no escape of mass in $\widehat{D}(\gamma_1)$ for the measures $\widehat{m}^L_{\text{Re}^{-1}(\gamma_1)} \to \widehat{m}_{\text{Re}^{-1}(\gamma_1)}$.\\

To avoid marking issues we do one further translation, moving our problem into $Q\mathcal{M}_{g,n}$. For this purpose the following lemma is very useful.\\

\begin{lemma}
	\label{no_marking_lemma}
	The quotient map $Q\mathcal{T}_{g,n}/\text{Stab}(\gamma_1) \to Q\mathcal{M}_{g,n}$ restricts to a homeomorphism from  $\text{Re}^{-1}(\gamma_1)/\text{Stab}(\gamma_1)$ onto $\mathbf{Re}^{-1}([\gamma_1])$. \\
\end{lemma}

\begin{proof}
	The restriction is clearly continuous, open, and surjective. It only remains to check it is injective. Suppose $[(X_1,q_1,\phi_1)], [(X_2,q_2,\phi_2)] \in \text{Re}^{-1}(\gamma_1)$ satisfy $[(X_1,q_1)] = [(X_2,q_2)]$ in $Q\mathcal{M}_{g,n}$. By definition, there exists a conformal diffeomorphism $I \colon X_1 \to X_2$ such that $I_* q_1 = q_2$. Let $[\varphi] := [\phi_2^{-1} \circ I \circ \phi_1] \in \text{Mod}_{g,n}$. The action of the mapping class $[\varphi]$ on $Q\mathcal{T}_{g,n}$ sends $[(X_1,q_1,\phi_1)]$ to $[(X_2,q_2,\phi_2)]$. Moreover, $[\varphi] \in \text{Stab}(\gamma_1)$. This proves the restriction is injective.
\end{proof}
$ $ \vspace{+0.15cm}

For every $L > 0$, let $\nu^L$ be the counting measure on $\text{Re}^{-1}(\text{Mod}_{g,n} \cdot \gamma_1)$ given by
\[
\nu^L := \frac{1}{L^{6g-6+2n}} \sum_{ \substack{[(X,q,\phi)] \in \text{Re}^{-1}(\text{Mod}_{g,n} \cdot \gamma_1), \\ \text{Im}([(X,q,\phi)]) \in \frac{1}{L} \cdot \mathcal{ML}_{g,n}(\mathbf{Z})}} \delta_{[(X,q,\phi)]}.
\]
Using real period coordinates one can check that the measures $\nu^L$ converge as $L \to \infty$ to a non-zero, locally finite measure $\nu$ on $\text{Re}^{-1}(\text{Mod}_{g,n} \cdot \gamma_1)$ which coincides with Lebesgue measure on real period coordinates. Let $\widehat{\nu}^L$ and $\widehat{\nu}$ be the local pushforwards of these measures under the quotient map $\text{Re}^{-1}(\text{Mod}_{g,n} \cdot \gamma_1) \to \mathbf{Re}^{-1}([\gamma_1])$. As taking local pushforward is a continuous operation, $\widehat{\nu}^L \to \widehat{\nu}$ as $L \to \infty$. Consider the subsets
\begin{align*}
E(\gamma_1) &:= \{[(X,q,\phi)] \in \text{Re}^{-1}(\text{Mod}_{g,n} \cdot \gamma_1)\ | \ \text{Area}([(X,q,\phi)]) \leq 1\}, \\
\widehat{E}(\gamma_1) &:= E(\gamma_1)/\text{Mod}_{g,n} \subseteq \mathbf{Re}^{-1}([\gamma_1]).
\end{align*}
By Lemma \ref{no_marking_lemma}, our problem translates to showing there is no escape of mass in $\widehat{E}(\gamma_1)$ for the measures $\widehat{\nu}^L \to \widehat{\nu}$, i.e. for every $\epsilon > 0$ we look for a compact subset $K_\epsilon \subseteq \widehat{E}(\gamma_1)$ with the following properties:
\begin{enumerate}
	\item $\widehat{\nu}(\partial K_\epsilon) = 0$.
	\item $\widehat{\nu}(\widehat{E}(\gamma_1)\backslash K_\epsilon) < \epsilon$.
	\item $\widehat{\nu}^L(\widehat{E}(\gamma_1) \backslash K_\epsilon) < \epsilon$ for all big enough $L > 0$.
\end{enumerate}
A nice feature of working in $\mathbf{Re}^{-1}([\gamma_1])$ is that we can do cut and paste operations on polygon representations without worrying about marking issues.\\

We now introduce our candidate $K_\epsilon$ sets. Let $K_\epsilon \subseteq \widehat{E}(\gamma_1)$ be the set of all quadratic differentials in the principal stratum of $\mathbf{Re}^{-1}(\gamma_1)$ having  area $\leq 1$ and whose horizontal saddle connections have length $\geq \epsilon$. To show the sets $K_\epsilon$ satisfy the desired properties we work in real period coordinates. To this end we begin by showing it is enough to consider finitely many real period coordinate charts. \\

\begin{lemma}
	\label{finitely_many_charts}
	The set $\mathbf{Re}^{-1}([\gamma_1])$ is covered by the image under the quotient map $\text{Re}^{-1}( \text{Mod}_{g,n} \cdot \gamma_1) \to \mathbf{Re}^{-1}([\gamma_1])$ of finitely many real period coordinate chart domains. \\
\end{lemma}

\begin{proof}
	Let $r$ be the number of connected components of the topological multi-curve underlying the integral multi-curve $\gamma_1$. Any point $[(X,\Sigma,q)] \in \mathbf{Re}^{-1}([\gamma_1])$ admits a polygon representation $\mathcal{P}$ given by $r$ horizontal parallelograms with singularities on their vertices, singularities on their top and bottom edges, and whose non-horizontal edges are identified by translations without negation. See Figure \ref{example} for an example when $\gamma_1$ is a simple closed curve. Notice that, up to changing the length of the saddle connections on the boundary of the parallelograms, there are only finitely many polygon representations $\mathcal{P}$ of this kind. Indeed, the number of singularities of a non-zero quadratic differential on $S_{g,n}$ is bounded above by $4g-4+2n$, the number of ways one can place these singularities on the top and bottom edges of $r$ parallelograms is finite, and the number of ways one can identify the resulting saddle connections on the boundary of the parallelograms is finite. In other words, $\mathbf{Re}^{-1}([\gamma_1])$ is covered by the image under the quotient map $\text{Re}^{-1}(\text{Mod}_{g,n} \cdot \gamma_1) \to \mathbf{Re}^{-1}([\gamma_1])$ of the domain of finitely many real period coordinate charts. 
\end{proof}
$ $ \vspace{+0.15cm}

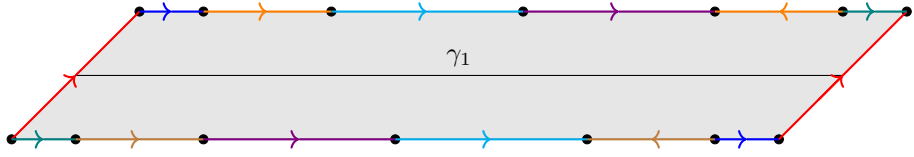
\begin{figure}
	\centering
	\vspace{+0.3cm}
	\begin{tikzpicture}[scale = 0.85]

\draw [fill = gray!20] (0,0) -- (2,2) -- (14,2) -- (12,0) -- (0,0);

\draw (1,1) -- (13,1) node[midway, above] {$\gamma_1$};

\draw [fill = black]  (0,0) circle [radius=2pt];
\draw [fill = black]  (2,2) circle [radius=2pt];
\draw [fill = black]  (14,2) circle [radius=2pt];
\draw [fill = black]  (12,0) circle [radius=2pt];

\draw [fill = black]  (3,2) circle [radius=2pt];
\draw [fill = black]  (5,2) circle [radius=2pt];
\draw [fill = black]  (8,2) circle [radius=2pt];
\draw [fill = black]  (11,2) circle [radius=2pt];
\draw [fill = black]  (13,2) circle [radius=2pt];

\draw [fill = black]  (1,0) circle [radius=2pt];
\draw [fill = black]  (3,0) circle [radius=2pt];
\draw [fill = black]  (6,0) circle [radius=2pt];
\draw [fill = black]  (9,0) circle [radius=2pt];
\draw [fill = black]  (11,0) circle [radius=2pt];

\draw[->,red, thick] (0,0) -- (1,1);
\draw[red, thick] (1,1) -- (2,2);
\draw[->,red, thick] (12,0) -- (13,1);
\draw[red, thick] (13,1,1) -- (14,2);

\draw[->,blue, thick] (2,2) -- (2.5,2);
\draw[blue, thick] (2.5,2) -- (3,2);
\draw[->,blue, thick] (11,0) -- (11.5,0);
\draw[blue, thick] (11.5,0) -- (12,0);

\draw[->,teal, thick] (0,0) -- (0.5,0);
\draw[teal, thick] (0.5,0) -- (1,0);
\draw[->,teal, thick] (13,2) -- (13.5,2);
\draw[teal, thick] (13.5,2) -- (14,2);

\draw[->,orange, thick] (3,2) -- (4,2);
\draw[orange, thick] (4,2) -- (5,2);
\draw[orange, thick] (11,2) -- (12,2);
\draw[<-, orange, thick] (12,2) -- (13,2);

\draw[->,brown, thick] (1,0) -- (2,0);
\draw[brown, thick] (2,0) -- (3,0);
\draw[brown, thick] (9,0) -- (10,0);
\draw[<-, brown, thick] (10,0) -- (11,0);

\draw[->,cyan, thick] (5,2) -- (6.5,2);
\draw[cyan, thick] (6.5,2) -- (8,2);
\draw[->,cyan, thick] (6,0) -- (7.5,0);
\draw[cyan, thick] (7.5,0) -- (9,0);

\draw[->,violet, thick] (8,2) -- (9.5,2);
\draw[violet, thick] (9.5,2) -- (11,2);
\draw[->,violet, thick] (3,0) -- (4.5,0);
\draw[violet, thick] (4.5,0) -- (6,0);

\end{tikzpicture}
	\vspace{+0.3cm}
	\caption{Example of a quadratic differential in the principal stratum of $\textbf{Re}^{-1}([\gamma_1]) \subseteq Q\mathcal{M}_{2,0}$ for a non-separating simple closed curve $\gamma_1$ in $S_{2,0}$.}
	\label{example}
\end{figure}

It is enough then to restrict our attention to a single one of the real period coordinate charts described in the proof of Lemma \ref{finitely_many_charts}. Let $\mathcal{P}$ be a polygon representation as above, representing quadratic differentials in an open subset $U$ of the principal stratum $\mathcal{S}$ of $\text{Re}^{-1}(\text{Mod}_{g,n} \cdot \gamma_1)$. To simplify vocabulary, we will make no distinction between $\mathcal{P}$ and the parallelograms in $\mathcal{P}$; in particular, when making reference to the boundary of $\mathcal{P}$ we will be refering to the boundary of the paralellograms in $\mathcal{P}$. Real period coordinates define a map $U \to W$ from the open subset $U \subseteq \mathcal{S}$ to an open subset $W \subseteq V$ of a vector subspace $V \subseteq \mathbf{R}^k$ of dimension $6g-6+2n$. This map assigns to every point in $U$ the value of the real part of the saddle connections on the boundary of $\mathcal{P}$, oriented counterclockwise. The subspace $V \subseteq \mathbf{R}^k$ describes the natural integral linear equations these real parts must satisfy to yield actual polygon representations. \\

Let $\rho_i$ denote the non-horizontal saddle connections on the boundary of $\mathcal{P}$. A priori, to cover all of $\mathbf{Re}^{-1}([\gamma_1])$ under the quotient map $\text{Re}^{-1}( \text{Mod}_{g,n} \cdot \gamma_1) \to \textbf{Re}^{-1}([\gamma_1])$, we may think that we need to consider all possible values $\text{Re}(\rho_i) \in \mathbf{R}$, but, as we are working on $\textbf{Re}^{-1}([\gamma_1])$, doing cut and paste operations on polygon representations yields the same quadratic differential, so we can actually restrict ourselves to $|\text{Re}(\rho_i)| \leq 1$. Taking this restriction into account, in the chart in consideration, the preimage under the quotient map of the set $\widehat{E}(\gamma_1)$ is a bounded set. Indeed, the condition $\text{Area}([(X,q)]) \leq 1$ is a linear inequality in real period coordinates which forces the absolute value of the real part of the  horizontal saddle connections on the boundary of $\mathcal{P}$ to be $\leq 1$. Similarly, the preimage of $K_\epsilon$ is described by the additional condition that the absoute value of the real part of the horizontal saddle connections on the boundary of $\mathcal{P}$ be $\geq \epsilon$. It follows that $K_\epsilon$ is compact and satisfies $\widehat{\nu}(\partial K_\epsilon) = 0$. Moreover, the preimage of $\widehat{E}(\gamma_1) \backslash K_\epsilon$ is a bounded set, whose boundary has Lebesgue measure zero, and whose Lebesgue measure is arbitrarily small for small values of $\epsilon>0$. A standard lattice counting argument in Euclidean space finishes the proof. For an example see Figure \ref{no_escape_of_mass_example}.\\

\begin{figure}
	\centering
	\begin{subfigure}[t]{0.45\textwidth}
		\centering
		\begin{tikzpicture}[scale = 0.85]

\draw [fill = gray!20] (0,0) -- (0,2) -- (4,2) -- (4,0) -- (0,0);

\draw [fill = black]  (0,0) circle [radius=2pt];
\draw [fill = black]  (0,2) circle [radius=2pt];
\draw [fill = black]  (2,2) circle [radius=2pt];
\draw [fill = black]  (4,2) circle [radius=2pt];
\draw [fill = black]  (4,0) circle [radius=2pt];
\draw [fill = black]  (2,0) circle [radius=2pt];

\draw[->,violet, thick] (0,0) -- (0,1);
\draw[violet, thick] (0,1) node[left] {$\rho$}  -- (0,2);
\draw[->,violet, thick] (4,0) -- (4,1);
\draw[violet, thick] (4,1) -- (4,2);

\draw[->,red, thick] (0,2) -- (1,2);
\draw[red, thick] (1,2)  node[above] {$\alpha$} -- (2,2);
\draw[red, thick] (2,2) -- (3,2);
\draw[<-, red, thick] (3,2) -- (4,2);

\draw[->,blue, thick] (0,0) -- (1,0);
\draw[blue, thick] (1,0) node[below] {$\beta$}  -- (2,0);
\draw[blue, thick] (2,0) -- (3,0);
\draw[<-, blue, thick] (3,0) -- (4,0);

\draw (0,1) -- (4,1) node[midway, above] {$\gamma_1$};

\end{tikzpicture} 
		\caption{Polygon representation.} \label{pol_rep_pillow}
	\end{subfigure}
	\hfill
	\begin{subfigure}[t]{0.45\textwidth}
		\centering
		\begin{tikzpicture}[scale = 0.85]

\draw [fill = gray!20] (0,0) -- (0,2) -- (4,2) -- (4,0) -- (0,0);

\draw [fill = blue!20] (0,0.5) -- (0,2) -- (4,2) -- (4,0.5) -- (0,0.5);

\draw (-0.5,0) -- (5,0) node[right] {$\text{Re}(\rho)$};

\draw (0,-0.5) -- (0,3) node[above] {$\text{Re}(\alpha)$};

\draw (4,2pt)--(4,-2pt) node[below] {$1$};

\draw (2pt,0.5)--(-2pt,0.5) node[left] {$\epsilon$}; 

\draw (2pt,2)--(-2pt,2) node[left] {$\frac{1}{2}$}; 

\end{tikzpicture}
		\caption{Real period coordinate chart.} \label{period_coord}
	\end{subfigure}

	\caption{No escape of mass property in the real period coordinate chart (\subref{period_coord}) associated to the polygon representation (\subref{pol_rep_pillow}), representing a flat pillowcase in the principal stratum of $\text{Re}^{-1}(\gamma_1) \subseteq Q\mathcal{T}_{4,0}$. The blue region covers $K_\epsilon$ and the gray region covers $\widehat{E}(\gamma_1) \backslash K_\epsilon$.}
	\label{no_escape_of_mass_example}
\end{figure}
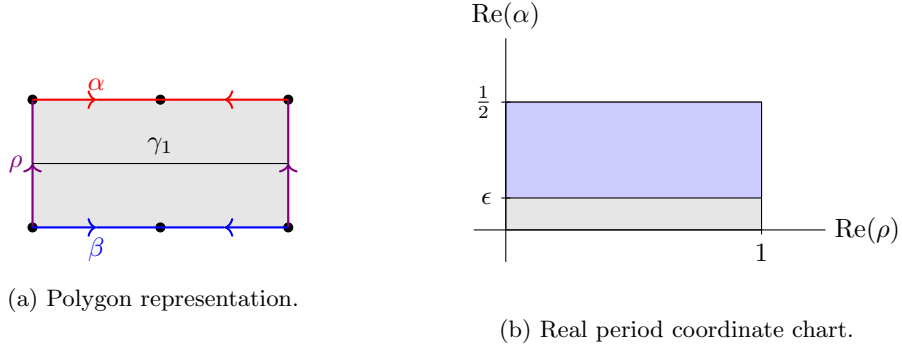

\textit{Proof of Theorem \ref{counting_quadratic_differentials_2_intro}.} Let $\gamma$ be an integral multi-curve on $S_{g,n}$. The same approach as above can be used to study the growth of the function $s(\gamma,*,L)$ defined in (\ref{counting_square_tiled_surfaces_2}) as $L \to \infty$. One only needs to replace the counting measures $\mu_{\gamma_2}^L$ in the arguments above by the counting measures $\mu^L$. As $\mu^L \to \mu_{\text{Thu}}$ when $L \to \infty$, Theorem \ref{counting_quadratic_differentials_2_intro} follows. \\

\section{Computing Thurston volumes}

For the rest of this section let $N := 3g-3+n$ and $\mathcal{P} := \{\gamma_1,\dots,\gamma_N\}$ be a pants decomposition of $S_{g,n}$. We compute $r_{g,n} := \frac{\widehat{\mu}_{\text{Thu}}(\widehat{B}(\mathcal{P}))}{c(\mathcal{P})}$. We refer to  $\widehat{\mu}_{\text{Thu}}(\widehat{B}(\mathcal{P}))$ as the \textit{Thurston volume} of $\mathcal{P}$.\\

\textit{Dehn-Thurston coordinates.} The following theorem, originally due to Dehn and rediscovered by Thurston in the context of measured foliations, gives an explicit parametrization of the set of integral multi-curves on $S_{g,n}$ in terms of their intersection numbers $m_i$ and their twisting numbers $t_i$ with respect to the curves $\gamma_i$ in $\mathcal{P}$. See \S 1.2 in \cite{PH92} for details.\\

\begin{theorem}
	\label{Dehn_Thurston_coordinates}
	There is a parametrization of $\mathcal{ML}_{g,n}(\mathbf{Z})$ by an additive semigroup $\Lambda \subseteq (\mathbf{Z}_{\geq 0} \times \mathbf{Z})^N$. The parameters $(m_i,t_i)_{i=1}^N \in (\mathbf{Z}_{\geq 0} \times \mathbf{Z})^N$ belong to $\Lambda$ if and only if the following conditions are satisfied:
	\begin{enumerate}
		\item For each $i =1,\dots, N$, if $m_i = 0$ then $t_i \geq 0$.
		\item For each complementary region $R$ of $S_{g,n} \backslash \mathcal{P}$, the parameters $m_i$ whose indices correspond to curves $\gamma_i$ of $\mathcal{P}$ bounding $R$ add up to an even number.\\
	\end{enumerate}
\end{theorem}

We refer to any parametrization as in Theorem \ref{Dehn_Thurston_coordinates} as a set of \textit{Dehn-Thurston coordinates} of $\mathcal{ML}_{g,n}(\mathbf{Z})$ adapted to $\mathcal{P}$ and to the additive semigroup $\Lambda\subseteq  (\mathbf{Z}_{\geq 0} \times \mathbf{Z})^N$ as the parameter space of such parametrization. For any set of Dehn-Thurston coordinates adapted to $\mathcal{P}$, the action of the full right Dehn twist along $\gamma_i$ on $\mathcal{ML}_{g,n}(\mathbf{Z})$ can be described in coordinates as $t_i \mapsto t_i + m_i$, leaving the other parameters constant. Dehn-Thurston coordinates will be extemely useful for computing the Thurston volume $\widehat{\mu}(\widehat{B}(\mathcal{P}))$.\\

\textit{Stabilizers of pants decompositions.} Let $\text{Stab}(\mathcal{P})$ be the stabilizer of $\mathcal{P}$ with respect to the $\text{Mod}_{g,n}$ action on $\mathcal{ML}_{g,n}$. It is a well known fact, see for instance \cite{Wol09a} and \cite{Wol09b}, that $\text{Stab}(\mathcal{P})$ is generated by:
\begin{enumerate}
	\item Full right Dehn twists along all curves $\gamma_i$ in $\mathcal{P}$.
	\item Half right Dehn twists along some curves $\gamma_j$ in $\mathcal{P}$.
	\item Finitely many finite order elements.\\
\end{enumerate} 

Let $[\phi] \in \text{Stab}(\mathcal{P})$ be a mapping class which does not fix every isotopy class of simple closed curves on $S_{g,n}$. For any set of Dehn-Thurston coordinates adapted to $\mathcal{P}$, one can check that the integral multi-curves stabilized by such mapping class correspond to an additive semigroup of positive codimension of the associated parameter space $\Lambda\subseteq  (\mathbf{Z}_{\geq 0} \times \mathbf{Z})^N$. \\

It may be the case that some mapping class $[\phi] \in \text{Mod}_{g,n}$ fixes every isotopy class of simple closed curves on $S_{g,n}$. The normal subgroup $K_{g,n} \lhd \text{Mod}_{g,n}$ of all such mapping classes has cardinality 
\begin{equation}
\label{epsilon_gn}
\epsilon_{g,n} := \left\lbrace
\begin{array}{ccl}
4 & \text{if} & (g,n) = (0,4), \\
2 & \text{if} & (g,n) \in \{(1,1),(1,2),(2,0)\}, \\ 
1 & \text{if} & (g,n) \notin  \{(0,4),(1,1),(1,2),(2,0)\}\\
\end{array} \right. .
\end{equation}
See 3.4 in \cite{FM11} for details.\\

It will be convenient for our purposes to consider the subgroup $\text{Stab}_*(\mathcal{P}) < \text{Stab}(\mathcal{P})$ generated by all full right Dehn twists along the curves $\gamma_i$ of $\mathcal{P}$. Notice that  $[\text{Stab}(\mathcal{P}):\text{Stab}_*(\mathcal{P})] < \infty$.\\

\textit{Thurston volumes of pair of pants decompositions.} We compute $\widehat{\mu}_{\text{Thu}}(\widehat{B}(\mathcal{P}))$ by reducing to a lattice counting problem on $\mathcal{ML}_{g,n}$ which we solve by considering Dehn-Thurston coordinates adapted to $\mathcal{P}$. It is convenient to work in the intermediate cover described in the following diagram:
\[
\begin{tikzcd}
\mathcal{ML}_{g,n}(\mathcal{P}) \arrow[dd] \arrow[dr] \\
& \mathcal{ML}_{g,n}(\mathcal{P})/\text{Stab}_*(\mathcal{P}) \arrow[ld, "p"]\\
\mathcal{ML}_{g,n}(\mathcal{P}) /\text{Stab}(P)
\end{tikzcd}.
\]
Let $\widetilde{\mu}_{\text{Thu}}$ be the local pushforward of the measure $\mu_{\text{Thu}}|_{\mathcal{ML}_{g,n}(\mathcal{P})}$ under the quotient map $\mathcal{ML}_{g,n}(\mathcal{P}) \to \mathcal{ML}_{g,n}(\mathcal{P})/\text{Stab}_*(\mathcal{P})$. The following proposition relates the measures $\widehat{\mu}_{\text{Thu}}$ and $\widetilde{\mu}_{\text{Thu}}$.\\

\begin{proposition}
	\label{relating pushforwards}
	Let $p_\# \widetilde{\mu}_{\text{Thu}}$ be the usual pushfoward of the measure $\widetilde{\mu}_{\text{Thu}}$ under the quotient map $p \colon \mathcal{ML}_{g,n}/\text{Stab}_*(P) \to \mathcal{ML}_{g,n}/\text{Stab}(P)$. Then
	\[
	p_\# \widetilde{\mu}_{\text{Thu}} = [\text{Stab}(\mathcal{P}) : \text{Stab}_*(\mathcal{P})] \cdot \widehat{\mu}_{\text{Thu}}.
	\]
\end{proposition}
$ $

\begin{proof}
	Let $\Omega_{g,n}(\mathcal{P}) \subseteq \mathcal{ML}_{g,n}(\mathcal{P})$ be the subset of all $\lambda \in \mathcal{ML}_{g,n}(\mathcal{P})$ satisfying the following conditions:
	\begin{enumerate}
		\item $\text{Stab}(\lambda) \cap \text{Stab}(\mathcal{P}) = K_{g,n}$. 
		\item $\text{Stab}(\lambda) \cap \text{Stab}_*(\mathcal{P}) = \{1\}$.
	\end{enumerate}
	As $K_{g,n} \lhd \text{Mod}_{g,n}$, it follows that $\Omega_{g,n}(\mathcal{P})$ is $\text{Stab}(\mathcal{P})$-invariant. Using Thurston's parametrization of $\mathcal{ML}_{g,n}$, see for instance 8.3.9 in \cite{Mar16}, one can check that $\Omega_{g,n}(\mathcal{P})$ is an open, full measure subset of $\mathcal{ML}_{g,n}(\mathcal{P})$. In particular, it is enough for our purposes to work on $\Omega_{g,n}(\mathcal{P})$.\\

	As local pushforwards are defined locally, we can restrict our attention to a single well covered open set of the form $U/K_{g,n} \subseteq \Omega_{g,n}(\mathcal{P})/\text{Stab}(\mathcal{P})$. Let $A \subseteq U$ be an arbitrary measurable set. By definition,
	\[
	\widehat{\mu}(A/K_{g,n}) = \frac{1}{\epsilon_{g,n}} \cdot \mu(A).
	\]
	Notice $p^{-1}(U/K_{g,n})$ can be written as a disjoint union of $[\text{Stab}(\mathcal{P}) : \text{Stab}_*(\mathcal{P})]/\epsilon_{g,n}$ images under the quotient map  $\mathcal{ML}_{g,n}(\mathcal{P}) \to \mathcal{ML}_{g,n}(\mathcal{P})/\text{Stab}_*(\mathcal{P})$ of $\text{Stab}(\mathcal{P})$-translates of $U$. Each one of these images is a well covered open set of the form $W/\{1\} \subseteq \mathcal{ML}_{g,n}(\mathcal{P})/\text{Stab}_*(\mathcal{P})$ for some $\text{Stab}(\mathcal{P})$-translate $W$ of $U$. In particular,
	\[
	p_\# \widetilde{\mu}_{\text{Thu}} (A/K_{g,n}) = \frac{[\text{Stab}(\mathcal{P}) : \text{Stab}_*(\mathcal{P})]}{\epsilon_{g,n}} \cdot \mu(A).
	\]
	We deduce
	\[
	p_\# \widetilde{\mu}_{\text{Thu}} (A/K_{g,n}) = [\text{Stab}(\mathcal{P}) : \text{Stab}_*(\mathcal{P})] \cdot \widehat{\mu}(A/K_{g,n}).
	\]
	As $A \subseteq U$ was an arbitrary measurable set, this finishes the proof.
\end{proof}
$ $ \vspace{+0.15cm}

Consider the subset
\[
\widetilde{B}(\mathcal{P}) := B(\mathcal{P})/\text{Stab}_*(\mathcal{P}) \subseteq \mathcal{ML}_{g,n}(\mathcal{P})/\text{Stab}_*(\mathcal{P}).
\]
It follows directly from Proposition \ref{relating pushforwards} that
\[
\widetilde{\mu}_{\text{Thu}}(\widetilde{B}(\mathcal{P})) = [\text{Stab}(\mathcal{P}): \text{Stab}_*(\mathcal{P})] \cdot \widehat{\mu}_{\text{Thu}}(\widehat{B}(\mathcal{P})).
\]
\vspace{+0.1cm}

Recall the definition of the counting measures $\mu^L$ on $\mathcal{ML}_{g,n}$ in (\ref{ML_counting_measure}). Recall that by definition $\mu^L \to \mu_{\text{Thu}}$ as $L \to \infty$. As restriction to open sets is a continuous operation, $\mu^L|_{\mathcal{ML}_{g,n}(\mathcal{P})} \to \mu_{\text{Thu}}|_{\mathcal{ML}_{g,n}(\mathcal{P})}$ as $L \to \infty$. Let $\widetilde{\mu}^L$ be the local pushforward of the measure $\mu^L|_{\mathcal{ML}_{g,n}(\mathcal{P})}$ under the quotient map $\mathcal{ML}_{g,n}(\mathcal{P}) \to \mathcal{ML}_{g,n}(\mathcal{P})/\text{Stab}_*(\mathcal{P})$. As taking local pushforward is a continuous operation, $\widetilde{\mu}^L \to \widetilde{\mu}_{\text{Thu}}$ as $L \to \infty$. Following the same no escape of mass arguments as in the previous section, we deduce $\widetilde{\mu}^L(\widetilde{B}(\mathcal{P})) \to \widetilde{\mu}_{\text{Thu}}(\widetilde{B}(\mathcal{P}))$ as $L \to \infty$. \\

Let $\mathcal{ML}_{g,n}(\mathbf{Z},\mathcal{P}) := \mathcal{ML}_{g,n}(\mathbf{Z}) \cap \mathcal{ML}_{g,n}(\mathcal{P})$. Notice $\text{Stab}_*(\mathcal{P})\cap \text{Stab}(\alpha) = \{1\}$ for all $\alpha \in \mathcal{ML}_{g,n}(\mathbf{Z},\mathcal{P})$. From (\ref{local_pushforward_atomic}) it follows that
\[
\widetilde{\mu}^L := \frac{1}{L^{6g-6+2n}} \sum_{[\alpha] \in \mathcal{ML}_{g,n}(\mathbf{Z},\mathcal{P}) /\text{Stab}_*(\mathcal{P})} \delta_{\frac{1}{L} \cdot [\alpha]}.
\]
From this we deduce
\[
\widetilde{\mu}^L(\widetilde{B}(\mathcal{P})) = \frac{\# \{[\alpha] \in \mathcal{ML}_{g,n}(\mathbf{Z},\mathcal{P}) /\text{Stab}_*(\mathcal{P}) \ | \ i(\alpha,\mathcal{P}) \leq L\}}{L^{6g-6+2n}}.
\]
\vspace{+0.2cm}

We now wish to count the number of points in the set
\[
\widetilde{I}_L := \{[\alpha] \in \mathcal{ML}_{g,n}(\mathbf{Z},\mathcal{P}) /\text{Stab}_*(\mathcal{P}) \ | \ i(\alpha,\mathcal{P}) \leq L \}.
\]
Considering Dehn-Thurston coordinates adapted to $\mathcal{P}$ with parameter space $\Lambda \subseteq (\mathbf{R}_{\geq 0} \times \mathbf{R})^N$, this is the same as counting the number of points in the set
\[
I_L:= \left\lbrace
\begin{array}{c | l}
(m_i,t_i)_{i=1}^N \in \Lambda
& \ m_i > 0, \ \forall i=1,\dots,N,\\
& \ 0 \leq t_i < m_i,  \ \forall i=1,\dots,N,\\
& \ \sum_{i=1}^N m_i \leq L.\\
\end{array} \right\rbrace.
\]
It follows that
\[
\lim_{L \to \infty} \widetilde{\mu}^L(\widetilde{B}(\mathcal{P})) = \lim_{L \to \infty} \frac{\# I_L}{L^{6g-6+2n}}.
\]
\vspace{+0.2cm}

Notice that the additive semigroup $\Lambda \subseteq (\mathbf{Z}_{\geq 0} \times \mathbf{Z})^N$ has index $2^{2g-3+n}$. Indeed, there is one even condition imposed on $\Lambda$ for every complementary region of $S_{g,n} - \mathcal{P}$, of which there are $2g-2+n$ in total, and one of these conditions is redundant. Standard lattice counting arguments in Euclidean space show that
\[
\lim_{L \to \infty} \frac{\# I_L}{L^{6g-6+2n}} = \frac{\text{Leb}(A_1)}{2^{2g-3+n}},
\]
where $\text{Leb}(A_1)$ is the standard Lebesgue measure of the set $A_1 \subseteq (\mathbf{R}_{\geq 0} \times \mathbf{R})^N$ given by
\[
A_1:= \left\lbrace
\begin{array}{c | l}
(x_i,y_i)_{i=1}^N \in (\mathbf{R}_{\geq 0} \times \mathbf{R})^N
& \ x_i > 0, \ \forall i=1,\dots,N,\\
& \ 0 \leq y_i < x_i,  \ \forall i=1,\dots,N,\\
& \ \sum_{i=1}^N x_i \leq 1.\\
\end{array} \right\rbrace.
\]
\vspace{+0.2cm}

Putting everything together we get
\[
\widehat{\mu}_{\text{Thu}}(\widehat{B}(\mathcal{P})) = \frac{\text{Leb}(A_1)}{[\text{Stab}(\mathcal{P}) : \text{Stab}_*(\mathcal{P})] \cdot 2^{2g-3+n}}.
\]
\vspace{+0.2cm}

\textit{Frequencies of pair of pants decompositions.} Mirzakhani's Weil-Petersson integration techniques in \cite{Mir07} can be used as in \cite{Mir08b} to show that
\[
c(\mathcal{P}) = \frac{\text{Leb}(A_1)}{[\text{Stab}(\mathcal{P}) : \text{Stab}_*(\mathcal{P})]}.
\]
\vspace{+0.2cm}

\textit{Conclusion.} From the computations above we deduce the following result.\\

\begin{theorem}
	For all $g,n \in \mathbf{Z}_{\geq 0}$ with $2 -2g- n < 0$,
	\[
	r_{g,n} = \frac{1}{2^{2g-3+n}}.
	\]
\end{theorem}
\vspace{+0.2cm}

\section{Examples}

The examples that follow are based on the work \cite{Mir08b} of Mirzakhani.\\

\textit{Genus 2 with no punctures.} Simple closed curves on $S_{2,0}$ are of one of two possible topological types: separating or non-separating. The frequency of separating simple closed curves $\gamma_1$ on $S_{2,0}$ is given by $c(\gamma_1) = \frac{1}{27648}$. The frequency of non-separating simple closed curves $\gamma_2$ on $S_{2,0}$ is given by $c(\gamma_2) = \frac{1}{576}$. By Theorem \ref{counting_quadratic_differentials_2_intro}, it follows that
\[
\lim_{L \to \infty} \frac{s(\gamma_1,*,L)}{s(\gamma_2,*,L)} = \frac{c(\gamma_1)}{c(\gamma_2)} = \frac{1}{48}.
\]
Roughly speaking, square-tiled surfaces in $Q\mathcal{M}_{2,0}$ with one horizontal cylinder and many squares have a $1/49$ chance of having separating horizontal core curve.\\

Analogously, by Theorem $\ref{counting_quadratic_differentials_intro}$, it follows that
\[
\lim_{L \to \infty} \frac{s(\gamma_1,\gamma_1,L)}{s(\gamma_2,\gamma_2,L)} = \frac{c(\gamma_1)^2}{c(\gamma_2)^2} = \frac{1}{2304}.
\]
Roughly speaking, square-tiled surfaces in $Q\mathcal{M}_{2,0}$ with one horizontal cylinder, one vertical cylinder, and many squares are $2304$ times more likely to have separating horizontal and vertical core curves rather than non-separating ones.\\

\textit{Genus 0 with punctures. } Let $\gamma_i$ be a simple closed curve on $S_{0,n}$ that cuts the surface into two disks, one with $i$ punctures and the other one with $n-i$ punctures. For simplicity, assume $2 < 2i < n$. The frequency of simple closed curves on $S_{0,n}$ of topological type $[\gamma_i]$ is given by
\[
c(\gamma_i) = \frac{1}{2^{n-4} (i-2)! (n-i-2)!(2n-6)}.
\]
By Theorem \ref{counting_quadratic_differentials_intro}, it follows that
\[
\lim_{L \to \infty} \frac{s(\gamma_i,*,L)}{s(\gamma_j,*,L)} = \frac{c(\gamma_i)}{c(\gamma_j)} = \frac{{n-4 \choose i-2}}{{n-4 \choose j-2}}.
\]
\vspace{+0.2cm}

\textit{Genus $g$ with no punctures.} Let $\gamma_i$ be a simple closed curve on $S_{g,0}$ that cuts the surface into two pieces, one of genus $i$ and the other of genus $g-i$. For simplicity, assume $2 < 2i < g$. The frequency of simple closed curves on $S_{g,0}$ of topological type $[\gamma_i]$  is given by
\[
c(\gamma_i) = \frac{1}{2^{3g-2}  24^g  i! (g-i)! (3i-2)!  (3(g-i)-2)!  (6g-6)}
\]
By Theorem \ref{counting_quadratic_differentials_intro}, it follows that
\[
\lim_{L \to \infty} \frac{s(\gamma_i,*,L)}{s(\gamma_j,*,L)} = \frac{c(\gamma_i)}{c(\gamma_j)} = \frac{{g \choose i}  {3g-4 \choose 3i-2}}{{g \choose j} {3g-4 \choose 3j-2}}.
\]
\vspace{+0.2cm}

\bibliographystyle{amsalpha}

\bibliography{bibliography}

\end{document}